\documentclass[aos]{imsart}

\RequirePackage{amsthm,amsmath,amsfonts,amssymb}
\RequirePackage[numbers, sort]{natbib}
\RequirePackage[colorlinks,citecolor=blue,urlcolor=blue]{hyperref}%% uncomment this for coloring bibliography citations and linked URLs
\RequirePackage{graphicx}%% uncomment this for including figures
\RequirePackage{array, tikz, subcaption}
\RequirePackage{enumitem}

\startlocaldefs
\numberwithin{equation}{section}
\theoremstyle{plain}
\newtheorem{theorem}{Theorem}[section]
\newtheorem{proposition}[theorem]{Proposition}
\newtheorem{lemma}[theorem]{Lemma}

\theoremstyle{definition}
\newtheorem{remark}{Remark}[section]

\newtheorem{assumption}[remark]{Assumption}
\newcommand{\Prob}{\mathbb{P}}

\newcommand\numberthis{\addtocounter{equation}{1}\tag{\theequation}}

\newcommand{\bs}[1]{\boldsymbol{#1}}

\makeatletter
\def\lstAZ{A, B, C, D, E, F, G, H, I, J, K, L, M, N, O, P, Q, R, S, T, U, V, W, X, Y, Z}
\def\lstaz{a, b, c, d, e, f, g, h, i, j, k, l, m, n, o, p, q, r, s, t, u, v, w, x, y, z}
\def\lstAZBB{B, C, D, E, F, G, H, I, J, K, L, M, N, O, Q, R, T, U, V, W, X, Y, Z}
\newcommand{\MkScr}[1]{\expandafter\def\csname s#1\endcsname{\mathscr{#1}}}
\newcommand{\MkUp}[1]{\expandafter\def\csname u#1\endcsname{\mathrm{#1}}}
\newcommand{\MkFrak}[1]{\expandafter\def\csname f#1\endcsname{\mathfrak{#1}}}
\newcommand{\MkCal}[1]{\expandafter\def\csname c#1\endcsname{\mathcal{#1}}}
\newcommand{\MkBB}[1]{\expandafter\def\csname #1#1\endcsname{\mathbb{#1}}}

\@for\i:=\lstAZ\do{%
	\expandafter\MkScr \i  %  
	\expandafter\MkFrak \i  %
	\expandafter\MkUp \i %
	\expandafter\MkCal \i  %
		  }    
\@for\i:=\lstaz\do{%
	\expandafter\MkUp \i   }    
	
\@for\i:=\lstAZBB\do{%
	\expandafter\MkBB \i     }
\makeatother
\newcommand{\PP}{\mathbb{P}}

\endlocaldefs

\begin{document}

\begin{frontmatter}
\title{On micromodes in Bayesian posterior distributions and their implications for MCMC}
%\title{A sample article title with some additional note\thanksref{T1}}
\runtitle{Micromodes and their implications for MCMC}
%\thankstext{T1}{A sample of additional note to the title.}

\begin{aug}
%%%%%%%%%%%%%%%%%%%%%%%%%%%%%%%%%%%%%%%%%%%%%%%
%% Only one address is permitted per author. %%
%% Only division, organization and e-mail is %%
%% included in the address.                  %%
%% Additional information such as            %%
%% identifying the corresponding author must %%
%% be included in in the Acknowledgments     %%
%% section if necessary.                     %%
\author[A]{\fnms{Sanket}~\snm{Agrawal}\ead[label=e1]{sanket.agrawal@math.ku.dk}},
\author[B]{\fnms{Sebastiano}~\snm{Grazzi}\ead[label=e2]{sebastiano.grazzi@unibocconi.it}}
\and
\author[C]{\fnms{Gareth O.}~\snm{Roberts}\ead[label=e3]{gareth.o.roberts@warwick.ac.uk}}
%%%%%%%%%%%%%%%%%%%%%%%%%%%%%%%%%%%%%%%%%%%%%%
%% Addresses                                %%
%%%%%%%%%%%%%%%%%%%%%%%%%%%%%%%%%%%%%%%%%%%%%%
\address[A]{University of Copenhagen\printead[presep={,\ }]{e1}}
\address[B]{Bocconi University\printead[presep={,\ }]{e2}}
\address[C]{University of Warwick\printead[presep={,\ }]{e3}}
\end{aug}

\begin{abstract}
We investigate the existence and severity of local modes in posterior distributions from Bayesian analyses. These are known to occur in posterior tails resulting from heavy-tailed error models such as those used in robust regression. To understand this phenomenon clearly, we consider in detail location models with Student-$t$ errors in dimension $d$ with sample size $n$. For sufficiently heavy-tailed data-generating distributions, extreme observations become increasingly isolated as $n \to \infty$. We show that each such observation induces a unique local posterior mode with probability tending to $1$. We refer to such a local mode as a \emph{micromode}. These micromodes are typically small in height but their domains of attraction are large and grow polynomially with $n$. We then connect this posterior geometry to computation. We establish an Arrhenius law for the time taken by one-dimensional piecewise deterministic Monte Carlo algorithms to exit these micromodes. Our analysis identifies a phase transition where a misspecified and overly underdispersed model causes exit times to increase sharply, leading to a pronounced deterioration in sampling performance. 
\end{abstract}

\begin{keyword}[class=MSC]
\kwd[Primary ]{62E20}
\kwd{62-08}
\kwd[; secondary ]{62F15}
\kwd{65C05}
\end{keyword}

\begin{keyword}
\kwd{robust Bayesian inference}
\kwd{outliers}
\kwd{extreme order statistics}
\kwd{multimodality}
\kwd{heavy-tailed data}
\kwd{large-sample asymptotics}
\kwd{Markov chain Monte Carlo}
\end{keyword}

\end{frontmatter}
%%%%%%%%%%%%%%%%%%%%%%%%%%%%%%%%%%%%%%%%%%%%%%
%% Please use \tableofcontents for articles %%
%% with 50 pages and more                   %%
%%%%%%%%%%%%%%%%%%%%%%%%%%%%%%%%%%%%%%%%%%%%%%
%\tableofcontents

%%%%%%%%%%%%%%%%%%%%%%%%%%%%%%%%%%%%%%%%%%%%%%
%%%% Main text entry area:

\section{Introduction}
    \label{sec:intro}
    Modern Bayesian methods are increasingly applied to large-sample settings with heavy-tailed models to accommodate atypical observations. Computation is often performed by generic algorithms that interact with the posterior distribution primarily through evaluations of the log-density and its derivatives. In such settings, the global properties of the posterior (e.g., tail behavior and concentration) may coexist with local geometric features induced by individual isolated observations. Understanding how these local features arise, how frequently they occur under heavy-tailed data-generating mechanisms, and how they affect the performance of Monte Carlo algorithms is the main motivation for this work.

    To analytically explore and study these phenomena in a tractable setting, let $y_1, \dots, y_n \in \mathbb{R}^d$ be an i.i.d. sample of size $n$ from some unknown distribution $P$. Consider the simple Bayesian location model for $P$,
    \begin{equation}
    \label{eq: linear model}
          y_j = x + \varepsilon_j, \qquad \varepsilon_1, \dots, \varepsilon_n \overset{\mathrm{i.i.d.}}{\sim} t_\nu
    \end{equation}
    where $x \in \mathbb{R}^d$ is an unknown location parameter, and $t_{\nu}$ is the standard $d$-dimensional Student-$t$ distribution with $\nu>0$ degrees of freedom. Such heavy-tailed models are common in Bayesian regression due to their robustness property against outliers  \citep[see, e.g., ][]{de1961bayesian, o2012bayesian, gramacy_shrinkage_2010, boonstra_multilevel_2021, lange_robust_1989, fernandez_multivariate_1999, fonseca_objective_2008, he_objective_2021, gagnon_theoretical_2023, de2025robust, gagnon2024fundamental}.  
    
    Let $\pi^{(n)}$ denote the Bayesian posterior for parameter $x$ under the model \eqref{eq: linear model} and the assumption of a non-informative prior. In dimension one, for $P = t_1$ and $\nu = 1$, \cite{reeds_asymptotic_1985} examined the distribution of local maxima of $\pi^{(n)}$. They showed that, as $n \to \infty$, every observation greater than $2n - \sqrt{n}$ results in exactly one local maximum within unit distance of it with probability $1$. It turns out that the existence of such local modes can be argued much more generally, for example in dimensions greater than one and whenever $P$ is sufficiently heavy-tailed. To illustrate this idea, take $d = 1$ and examine the first two derivatives of the negative log-posterior
\begin{equation}
    \label{eq:derivatives}
            \frac{-\partial\log \pi^{(n)}(x)}{\nu + 1} = \sum_{j=1}^n \frac{x - y_j}{\nu + (x - y_j)^2}; \quad \frac{-\partial^2\log \pi^{(n)}(x)}{\nu + 1} = \sum_{j=1}^n \frac{(\nu - (x - y_j)^2)}{(\nu + (x - y_j)^2)^2}.
    \end{equation}
    Suppose the observations are configured such that one of them, say $y_n$ without loss of generality, is remote and isolated from the rest of the dataset, i.e., $|y_n - y_j| \approx |y_n| \gg 1$ for all $j \le n-1$. Then, for $x$ close to $y_n$, we may approximate the terms in \eqref{eq:derivatives} by
    \begin{equation}
    \label{eq:approx_intro}
        \frac{-\partial\log \pi^{(n)}(x)}{\nu + 1} \approx \frac{n-1}{x} + \frac{(x - y_n)}{\nu + (x - y_n)^2}; \ \frac{-\partial^2\log \pi^{(n)}(x)}{\nu+1} \approx -\frac{n-1}{x^2} + \frac{(\nu - (x - y_n)^2)}{(\nu + (x - y_n)^2)^2}.
    \end{equation}
    If $y_n \gg n$, then $n/|x| \approx n/|y_n| \ll 1$ and the first term in both expressions above becomes negligible. As a result, the gradient and the curvature of the posterior $\pi^{(n)}$ near $y_n$ are influenced solely by $y_n$. 
    This causes an unusual bump in the posterior density which manifests as a local maximum.
    
    For a sufficiently heavy-tailed true distribution $P$, such configurations as described above are not atypical. They become more common as either the size of the dataset, $n$, or the dimension, $d$, increases. In either case, the outlying observations become more remote and result in a local maximum around them. We call these local maxima \emph{micromodes} as they are local modes of the posterior distribution and, although big within their domain of attraction, they are tiny in absolute height. See the top panels of Figure \ref{fig:micromodes_and_target}, for an illustration in $d=2$.  

    \begin{figure}[t]
        \centering
        \includegraphics[width=0.49\linewidth, height=0.44\linewidth]{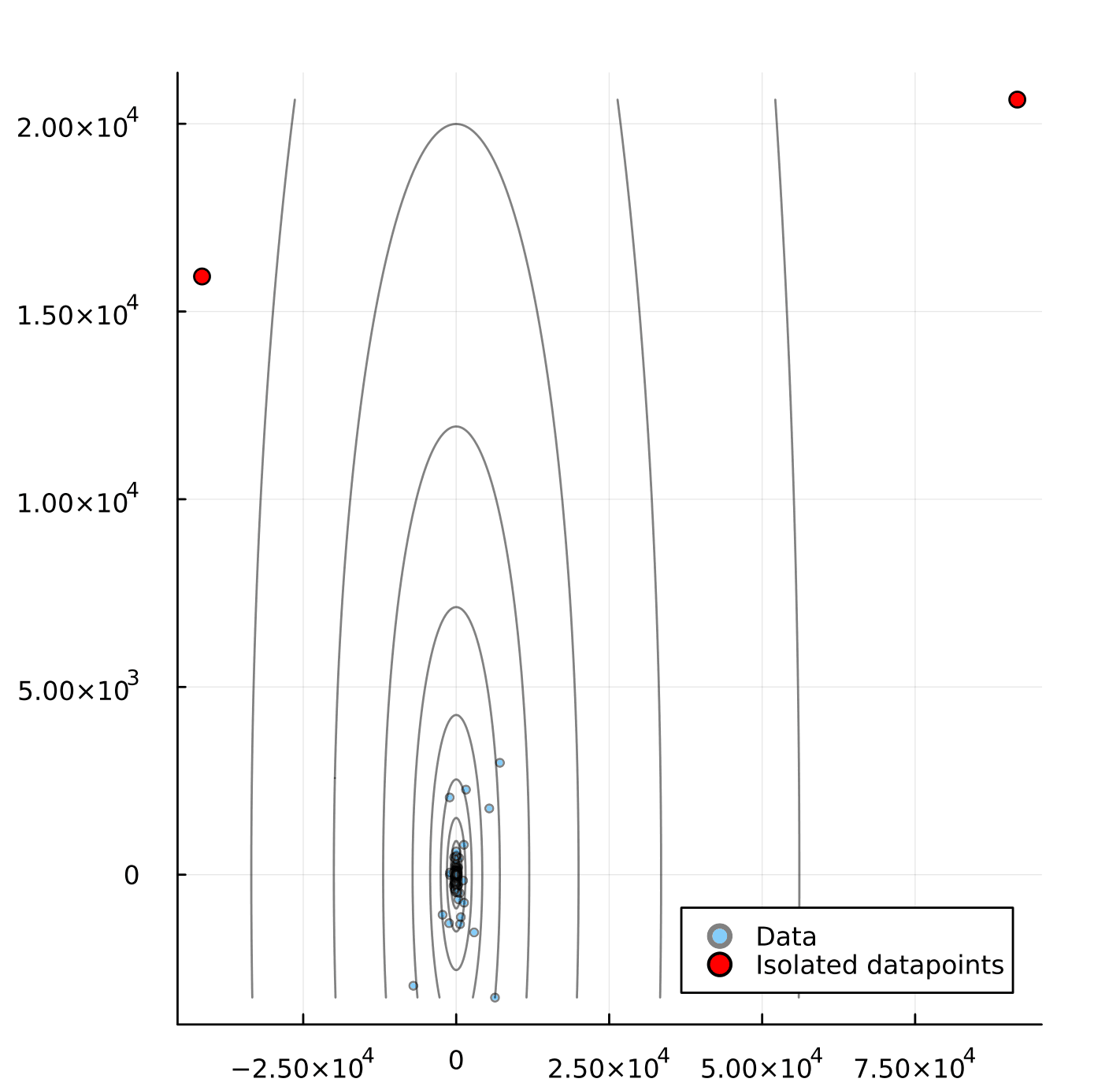}
        \includegraphics[width=0.49\linewidth, height=0.44\linewidth]{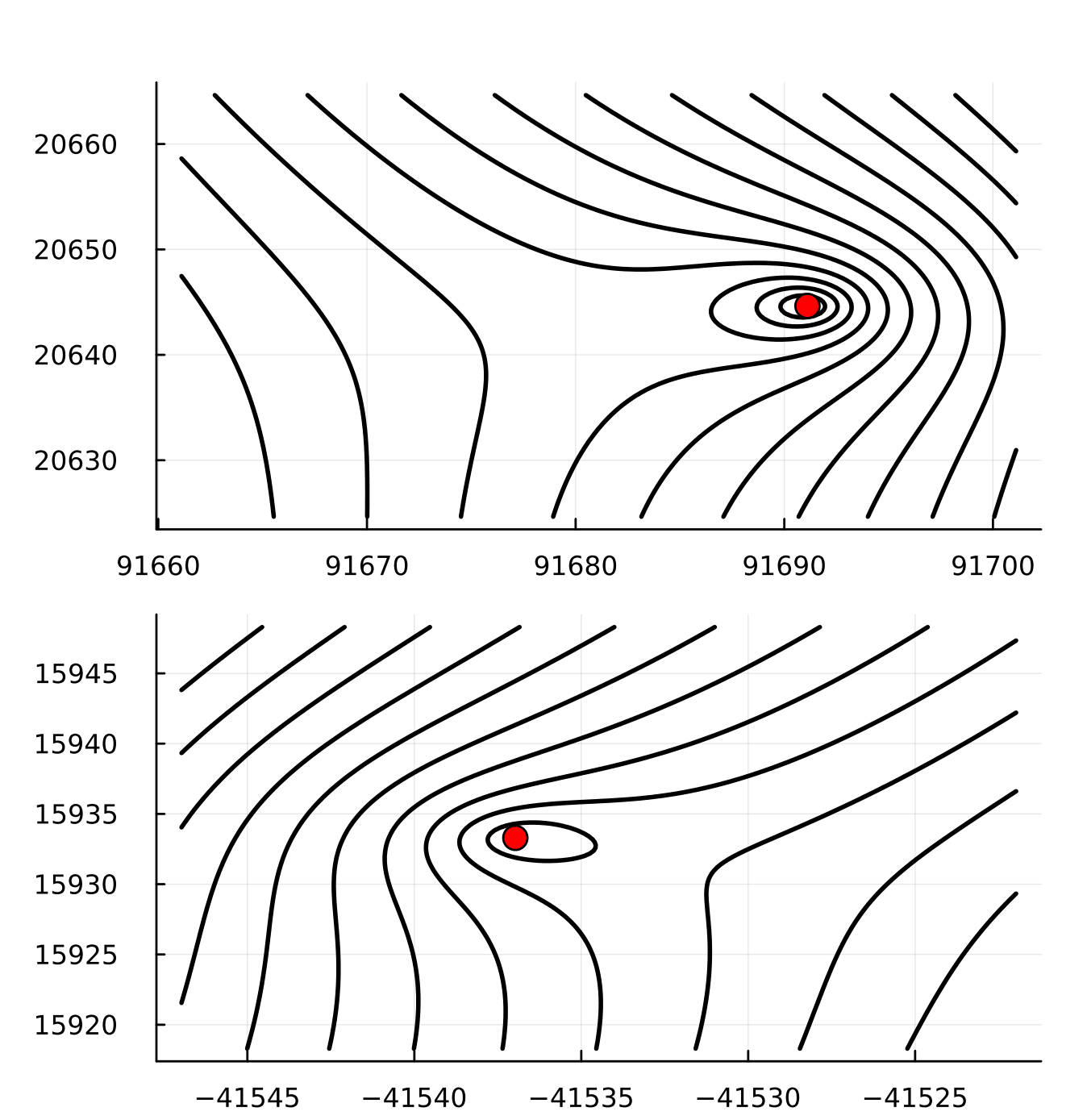}
        \includegraphics[width=0.49\linewidth, height=0.44\linewidth, trim={0 0 10pt 0}, clip]{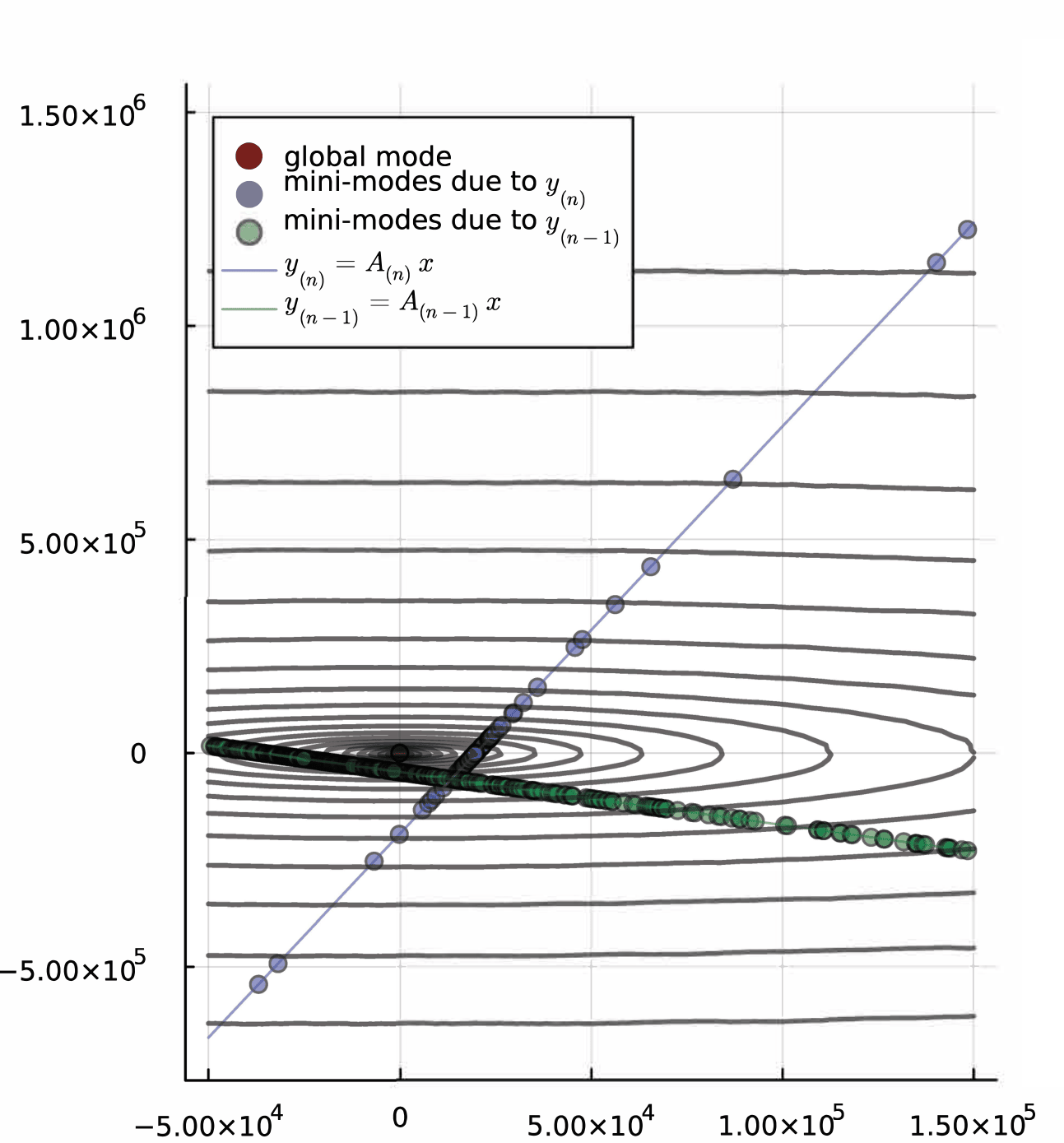}
        \includegraphics[width=0.49\linewidth, height=0.44\linewidth, trim={0 0 10pt 0}, clip]{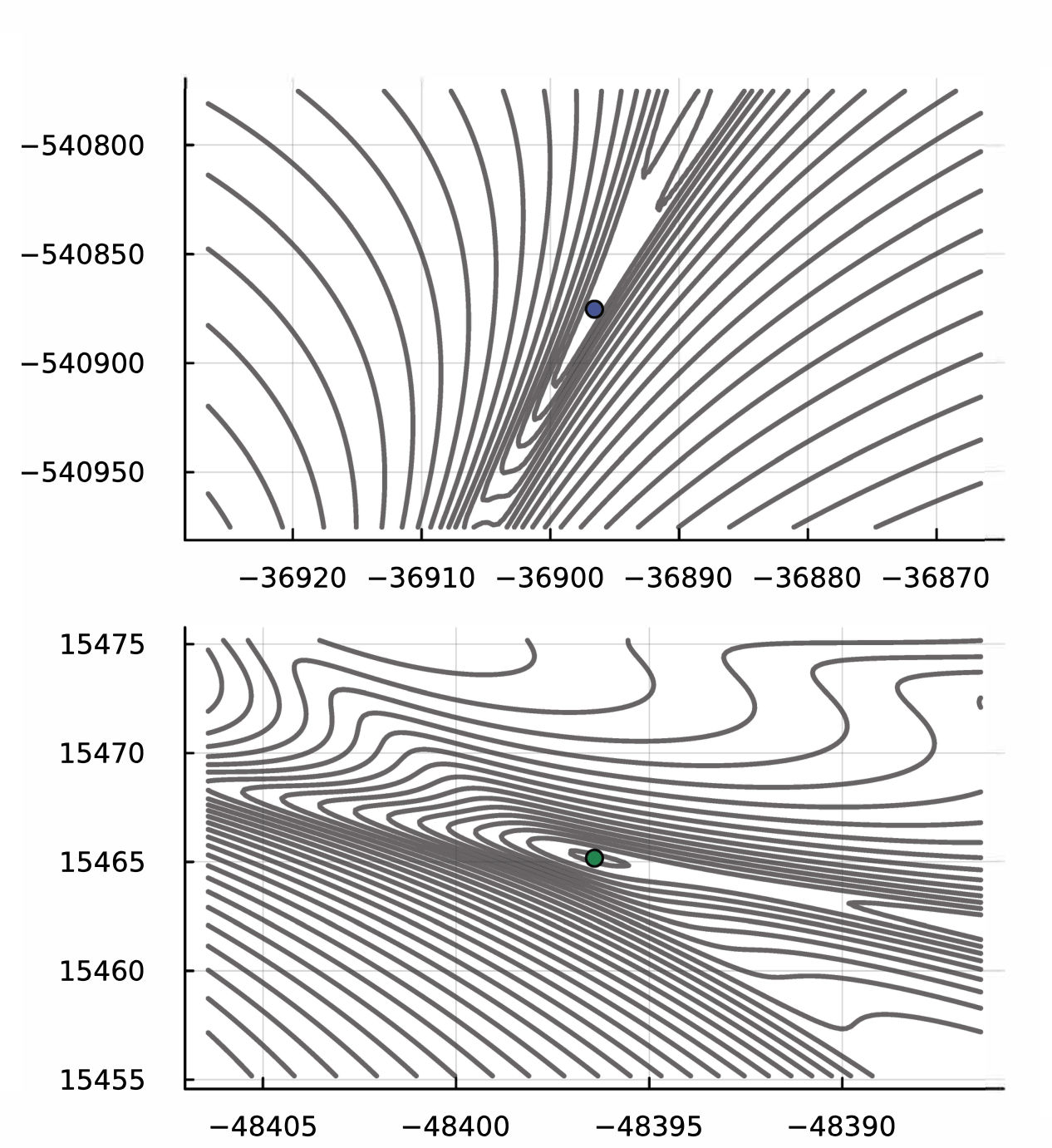}
        \caption{Isodensity contours of 2-dimensional posterior distributions with $n=10^5$ datapoints. Top panels for the model in \eqref{eq: linear model} with $d=2$. Blue and red dots for the datapoints. Bottom panels for the linear model $y_j = A_jx +  \varepsilon_j$, $d = 2$ covariates and $n = 10^5$ datapoints. Blue and green dots for the mini-modes found along the lines $y_{i} = A_{i}x$ where $y_i$ is the largest and second largest observation. For both models, $\varepsilon_j \sim t_1$ and $\nu = 1$. Right panels: zoom-in near micromodes.}
        \label{fig:micromodes_and_target}
    \end{figure}

    Micromodes are a nuisance for most computational algorithms used to fit Bayesian models. They create an obvious bottleneck for optimization methods which may fail to converge to the global mode. They also pose challenges for Markov chain Monte Carlo (MCMC; see \citep[]{Roberts_2004} for an overview) methods which are traditionally known to become trapped in local modes \citep[see, e.g.,][]{geyer_annealing_1995, brooks_efficient_2003}. 
    In this work we focus on the Zig-Zag process (ZZP) of \cite{bierkens_zig-zag_2019}, which is a continuous-time non-reversible piecewise deterministic Monte Carlo algorithm. ZZP has attracted recent attention from both theorists and practitioners as  a promising method for accelerating the convergence of Monte Carlo estimators \citep{diaconis2000analysis, gagnon2024asymptotic, ascolani2025fast, sun2010improving, neal2004improving, bierkens2016non, bierkens2023sticky} and support its use for multi-modal contexts \cite{monmarche_piecewise_2016, vasdekis_note_2022}. This algorithm is designed to mitigate diffusive random-walk behavior and instead  promote ballistic dynamics and rapid mixing.
    Furthermore, ZZP possesses the property of principled sub-sampling \citep{bierkens_zig-zag_2019}, whereby, at each iteration of the algorithm, a randomly chosen subset of data can be used to estimate the likelihood without biasing the MCMC output. From a computational perspective, this results in an algorithm whose overall computational cost  grows sub-linearly with the size of the dataset \citep{agrawal_large_2024}, rendering Zig-Zag with subsampling a scalable and more efficient alternative to other traditional methods.

    While Bayesian analysis of heavy-tailed models is well-studied in the literature, the emergence of such micromodes has been largely neglected in the robust Bayesian statistics and Bayesian asymptotics literature. In this paper, we take a step towards an integrated analysis that links the posterior geometry with the behavior of the computational method in this setting. The main contributions of this article are as follows: 
    \begin{enumerate}
        \item  We provide a rigorous analysis of the emergence and identification of micromodes in the posterior $\pi^{(n)}$ on $\RR^d$. Under the assumption $P$ has polynomial tails of order $\beta > 0$, extreme order statistics grow remote at scale $n^{1/\beta}$. We show that for each such extreme observation, say $Y$, the posterior admits (with probability tending to $1$) a \emph{unique} micromode within a $\sqrt{\nu}$-neighborhood of $Y$ (Theorem~\ref{theo:maxima}). We obtain explicit bounds for the location of the micromode and characterize its geometry via its width $W_n$, showing $W_n$ scales polynomially as $n^{1/\beta-1}$ (Theorem~\ref{theo:width}). The proof combines results from extreme value theory with a high-precision approximation of the empirical score function in remote regions of space(Lemma~\ref{lemma:score_bound} and Theorem~\ref{theo:score_bound}). This reduces micromode formation to a tractable local root-finding problem with a strongly convex surrogate potential.
        \item We use the results above to study the time required for a ZZP to escape from the micromodes in one dimension. Using standard renewal arguments, we show that the expected exit time is of the order $n^{(1/\beta - 1)(\nu + 1)}$ for both canonical ZZP and subsampling ZZP (Theorem~\ref{theo:arrhenius_law}). These results indicate that ZZP with subsampling can achieve comparable exit-times, while being computationally cheaper than its canonical counterpart.
        \item Our analysis sheds light on the interaction between the assumed model \eqref{eq: linear model} and the data-generating distribution $P$, and its implication for an MCMC algorithm. In particular, we show that a mismatch in the tail decay can alter the posterior landscape adversely. This leads to a phase transition where an overly light-tailed (underdispersed) assumed model can cause exit times of ZZP to increase sharply, producing a pronounced deterioration in sampling performance.
    \end{enumerate}

     Our analysis focuses on the simple model in \eqref{eq: linear model}, however the qualitative mechanism underlying micromode formation is not specific to this setting. Indeed, similar phenomena arise much more generally whenever (i) individual observations are sufficiently isolated and (ii) the likelihood is heavy-tailed. This is illustrated in the bottom panels of Figure~\ref{fig:micromodes_and_target} for the regression model $y_i = A_i x + \varepsilon_i$, with univariate heavy-tailed noise $\varepsilon_i \sim t_1$ and generic covariates $A_i$. Here, hundreds of micromodes are numerically discovered along the hyperplanes $\{x \colon y_i = A_i x\}$, for the two most extreme observations $y_i$.

Although our work is theoretical, we believe there are important methodological implications when using heavy-tailed models. First, micromodes are likely to be present. Although these modes will contain exponentially small (in $n$) posterior mass, their large width can cause havoc for MCMC (and other) algorithms. In practice, it may often be very difficult to start an MCMC algorithm from the global mode of the posterior. However, the worst effects of the micromodes can be mitigated by trying to ensure that the models chosen allow error distributions {\em at least as heavy} as the actual data distribution.

    \subsection{Related literature}
        Our work was partially inspired by \cite{reeds_asymptotic_1985}, who examined the distribution of micromodes when both the assumed model and the true distribution are $t_1$ (i.e. Cauchy) in dimension one. 
        Theorem~\ref{theo:maxima} resonates with (and in some aspects generalizes) these results by considering more general heavy-tailed data in dimension $d$ and by allowing for model misspecification. In general, finding the roots of the likelihood equation for model \eqref{eq: linear model} is a long-standing problem in the literature \cite[see e.g.,][]{ferguson_maximum_1978, fonseca_objective_2008, barnett_evaluation_1966}. Our analysis provides a step forward in the direction of identifying and characterizing these roots.

        Existing analysis of MCMC largely focuses on uni-modal and log-concave distributions \citep{chewi2023log, gelman2008weakly}. We consider posterior distributions which are outside this class. The analysis of Markov processes on multimodal targets has historically been studied under the framework of metastability \citep{bovier2016metastability}.
        Our MCMC analysis in Section~\ref{sec:ek-formula} expands the work of \cite{monmarche_piecewise_2016, peutrec_eyringkramers_2025}, who proved an Arrhenius law for crossing time of energy barriers for ZZP in a double-well potential example. In particular, they showed that, when the two wells remain a fixed distance apart, the expected exit time on the log-scale scales at the same rate as the height (or depth) $\varepsilon^{-1}$ of the modes as $\varepsilon \to \infty$. In contrast, in our setting, the energy barrier arises not only from the height of the micromode, but also from its width. Finally, in a similar spirit to ours, \cite{bandeira2023free} established an energy-barrier result for gradient-based MCMC in high-dimensional Bayesian inverse problems with cold starts.

    \subsection{Organization of the paper}
        The remainder of the paper is organized as follows. Section~\ref{sec: Assumptions and statement of main results} contains the main results of the paper.  
        Section~\ref{sec:body_micromode} deals with the existence and characterization of micromodes depending on $\beta$ (see Assumption~\ref{assum:density}). 
        Section~\ref{sec:ek-formula} presents an Arrhenius law for the exit time of ZZP from a micromode in one dimension. 
        Section~\ref{sec:essential} discusses the implications of the above results and the deterioration in algorithm performance resulting from the mismatch between $\nu$ and $\beta$.

        Section~\ref{sec:score_approximation}  formalizes, in a general setting, the heuristics presented in \eqref{eq:approx_intro} and approximates the empirical score via a law of large numbers in remote locations of the space with high precision. 
        Section~\ref{sec:micromode_technical} presents the proofs of Theorems~\ref{theo:maxima}--\ref{theo:width}, establishing the existence and width of micromodes. Section \ref{sec:exit_time_body} presents a proof of Theorem~\ref{theo:arrhenius_law}. A discussion on possible extensions of this work is given in Section~\ref{sec:limitations}. More technical statements and proofs are deferred to the appendix.

\section{Main results}\label{sec: Assumptions and statement of main results}
    \subsection{Existence and characterization of micromodes}
    \label{sec:body_micromode}
        Let $(\Omega, \mathcal{F}, \PP)$ be an arbitrary probability space rich enough to define all the following random quantities. For all $n$, $\{Y_1, \dots, Y_n\}$ is a random dataset of size $n$ in $\mathbb{R}^d$ distributed independently and identically according to a probability distribution $P$. For a fixed $k$, the random vector with $k$-th smallest radius among $Y_1, \dots, Y_n$ is denoted by $Y_{(k)}$. The corresponding distribution of $Y_{(k)}$ is denoted by $P_{(k)}$. The distribution $P$ satisfies the following heavy-tailed assumption.
        \begin{assumption}\label{assum:density} 
        Let $|\cdot|$ denote the Euclidean norm in $\mathbb{R}^d$. The distribution $P$ is a probability distribution on $\RR^d$ with density $p(y) = h(|y|)$ where $h:(0, \infty) \to (0, \infty)$ is locally bounded and such that $\lim_{t \downarrow 0}h(t)$ is positive and finite. The function $h$ is of {\it regular variation}\footnote{A measurable function $h:(0, \infty) \to (0, \infty)$ is of  regular variation with index $\rho \in \mathbb{R}$ if, for all $r > 0$, $\lim_{t \to \infty} h(r t)/h(t) = r^{\rho}$. In such case, we write  $h \in RV_{\rho}$.} with index $-(\beta + d)$ for some $\beta > 0$ and satisfies $\lim_{t \uparrow \infty} t^{\beta+d}h(t) = K(\beta, d)$ for some constant $K(\beta, d) > 0$. Finally, there exists a $t_1$ such that $h$ is non-increasing on $[t_1, \infty)$.
        \end{assumption}

        Roughly speaking, $P$ is a $d-$dimensional isotropic distribution whose radial density exhibits polynomial tail of order $\beta$. Although Assumption~\ref{assum:density} might seem restrictive, it is a common assumption for analyzing heavy-tailed data, see e.g. \cite{resnick_extreme_2008}. Note also that standard multivariate Student-$t$ distributions with $\beta>0$ degrees of freedom satisfy Assumption~\ref{assum:density}. 
        
        Under the parametric Bayesian model \eqref{eq: linear model} for $P$ and the assumption of a non-informative prior, the posterior distribution for the parameter $x$, denoted by $\pi^{(n)}$, is given by,
        \begin{equation}
        \label{eq:regression_post}
            \pi^{(n)}(x) \propto \prod_{j=1}^{n} \frac{1}{(\nu + |x - Y_j|^2)^{(\nu + d)/2}}, \quad x \in \mathbb{R}^d.
        \end{equation} 
        Let $S(x,y) := (x-y)/(\nu + |x-y|^2)$. We use $S_n$ to denote the empirical score
        \begin{equation}
        \label{eq:empirical_score}
            S_n(x) := \frac{-1}{\nu +d} \nabla_x \log\pi^{(n)}(x) = \sum_{j=1}^{n} S(x, Y_{j}), \quad x \in \mathbb{R}^d.
        \end{equation}
        Similarly  
        \begin{equation}
        \label{eq:empirical_info}
            S'_n(x) := \frac{-1}{\nu +d} \nabla^2_x \log\pi^{(n)}(x) = \sum_{j=1}^{n} S'(x, Y_{j}), \quad x \in \mathbb{R}^d,
        \end{equation}
        denotes the empirical information matrix.
        
        Finding the local modes of the posterior $\pi^{(n)}$ defined in \eqref{eq:regression_post} is equivalent to finding the roots of $S_n(x) = 0$ for which the second-order condition $S'_n(x) \succ 0$, i.e. $S'_n(x)$ is positive definite, holds. Repeating the argument in \cite{reeds_asymptotic_1985}, it can be shown  that any local maximum of $\pi^{(n)}$ must lie within a ball with radius $\sqrt{\nu}-$ of one of the datapoints. However, not all points will carry a local maximum in their neighborhood. As argued in the introduction, the ability of a datapoint to manifest a local maximum will depend on how isolated it is from the rest of the dataset. One of the immediate consequences of Assumption \ref{assum:density} is that the extreme values in the dataset scale in their radial component at a polynomial rate of $n^{1/\beta}$.
        \begin{proposition}
        \label{propo:evt}
          Suppose $P$ satisfies Assumption \ref{assum:density} for some $\beta > 0$. 
          Then, for each fixed $k \ge 0$, there exists a $G_k$ which follows a Gamma distribution with shape parameter $k+1$ and unit scale, such that
          \begin{equation*}
          \label{eq:evt1}
            \frac{|Y_{(n-k)}|}{n^{1/\beta}} \overset{d}{\longrightarrow}  \left(\frac{G_k}{A}\right)^{-1/\beta}, 
          \end{equation*}
          as $n \to \infty$, where $A = \frac{2\pi^{d/2}K(\beta, d)}{\beta\ \Gamma(d/2)}$ and $\Gamma(\cdot)$ denotes the gamma  function. 
          
          Moreover, for all $0 < \epsilon < 1/\beta$, as $n \to \infty$,
            \begin{equation*}
            \label{eq:evt2}
                \PP\left(\min_{j \ne {n-k}}|Y_{(j)} - Y_{(n-k)}| > |Y_{(n-k)}|/2n^{\epsilon}\right) \longrightarrow 1.
            \end{equation*}
        \end{proposition}
        To make our constructions explicit in the remainder of this paper, we suppose that the distributional convergence in Proposition~\ref{propo:evt} holds almost-surely. Note that this is possible due to Skorokhod's representation theorem. 
        
        Proposition \ref{propo:evt} implies that for each fixed $k$, the point with the $k-$th biggest radius, $Y_{(n-k)}$, becomes more remote and isolated from the rest of the dataset at a rate $n^{1/\beta}$. From \eqref{eq:approx_intro}, the score contribution of the rest of the dataset in a neighborhood of $Y_{(n-k)}$ may then be approximated by $n^{1-1/\beta}$. When $\beta > 1$, $n^{1-1/\beta}$ increases to infinity. Thus, the isolation of $Y_{(n-k)}$ is not enough for a micromode to emerge within the $\sqrt{\nu}-$neighborhood of $Y_{(n-k)}$, as implied by the next theorem. Let $B_r(y)$ denote the ball of radius $r$ centered at $y$.
        \begin{theorem}
        \label{theo:no_roots}
            Suppose Assumption~\ref{assum:density} holds with $\beta > 1$. Let $k \ge 0$ be fixed. With probability going to $1$ as $n \to \infty$, for all $x \in B_{\sqrt{\nu}}(Y_{(n-k)})$, there exists a $v \in \mathbb{S}^{d-1}$ such that $v^TS_n(x) < 0$.
        \end{theorem}

        On the other hand, when $\beta \le 1$, the first term in \eqref{eq:approx_intro} becomes negligible as $n$ increases. In particular, as long as $|Y_{(n-k)}| > 2n\sqrt{\nu}$, the likelihood contribution of $Y_{(n-k)}$ in its $\sqrt{\nu}-$neighborhood dominates the rest of the dataset and results in a micromode. 
        \begin{theorem}
        \label{theo:maxima}
            Suppose Assumption \ref{assum:density} holds with $0<\beta\le 1$. Let $k \ge 0$ be fixed. For $n > k$, let $E_n$ be the event defined in \eqref{eq:en_sets}. There exists an $N_0 \ge 4$ such that for all $n > N_0$, on the event $E_n$, there exists a unique local maximum $x_n^{+}$ of $\pi^{(n)}$ in the $\sqrt{\nu}-$neighborhood of $Y_{(n-k)}$. Additionally, $x_n^{+}$ satisfies, 
            \begin{equation}
            \label{eq:cbeta}
                |Y_{(n-k)} - x_n^{+}| \le \left(\frac{4\nu}{G_k+ c_{\beta}}\right)n^{1-1/\beta}; \qquad c_{\beta} = \begin{cases}
                    0, & 0 < \beta < 1,\\
                    2\sqrt{\nu}, & \beta = 1.
                \end{cases},
            \end{equation}
            where $G_k$ is as in the statement of Proposition~\ref{propo:evt}. Moreover, define $A_n := \{|Y_{(n-k)}| > 2n\sqrt{\nu}\}$. Then, $E_n \subset A_n$ for large $n$. And, as $n \to \infty$,
            \begin{equation*}
            \label{eq:lim_En}
                \PP(E_n \mid A_n) \to 1, \quad \PP(A_n) \longrightarrow \begin{cases}
                    1, & 0 < \beta < 1,\\
                    \PP(G_k > 2\sqrt{\nu}), & \beta = 1.
                \end{cases}
            \end{equation*}
        \end{theorem}
        Theorem~\ref{theo:maxima} is a consequence of the scaling of $Y_{(n-k)}$ given by Proposition~\ref{propo:evt} and the approximation of the empirical score $S_n$ via a law of large numbers in Lemma~\ref{lemma:score_bound}. Given that $|Y_{(n-k)}| > 2n\sqrt{\nu}$, it implies the existence of exactly one local maximum within a $\sqrt{\nu}-$distance of $Y_{(n-k)}$ with limiting probability $1$. This is easily extended to a collection of any finite $K$ extreme values $\{Y_{(n - k_1)}, \dots, Y_{(n - k_K)}\}$ via a simple union bound and the joint weak convergence of extreme order statistics \citep[see][Theorem 2.1.1]{de_haan_extreme_2006}. Theorem~\ref{theo:maxima} generalizes \cite{reeds_asymptotic_1985} to dimensions greater than one, to more general data distributions, and to possibly misspecified models (here $\nu$ vs $\beta$), which was earlier stated only for the one-dimensional Cauchy distribution in a well-specified model. However, it is not a complete generalization since Theorem~\ref{theo:maxima} only extends to finitely many points and hence does not cover all observations greater than $2n\sqrt{\nu}$ although, see also the discussion following Lemma~\ref{lemma:score_bound}.

        Let $k$ be fixed. For large $n$, let $x_n^{+}$ be the unique local maximum of $\pi^{(n)}$ in the neighborhood of $Y_{(n-k)}$ guaranteed by Theorem~\ref{theo:maxima}. Then, $S_n(x_n^{+}) = 0$ and $S'_n(x_n^{+})$ is positive definite. The continuity of $S_n$ implies that there exists a $t > 0$ such that $v^T S_n(x + vu) > 0$ for all $v \in \mathbb{S}^{d-1}$ and $0 < u < t$. This means starting from $x_n^{+}$, the posterior density decreases monotonically in all directions for a distance $t$. We define the \emph{width} of the micromode associated with $Y_{(n-k)}$ to be the largest such $t$ and denote it by $W_n$. More precisely, we define,
        \begin{equation}
        \label{eq:width}
            W_n := \sup\left \{t \ge 0: v^TS_n(x_{n}^{+} + vt) > 0\  \forall \ v \in \mathbb{S}^{d-1}\right\}.
        \end{equation}
        Note that the notion of width of a local maximum $x$ of an arbitrary function $f$ can be more generally defined. Loosely speaking, it is the largest distance up to which, starting from $x$, $f$ is monotonically decreasing in all directions. See Figure~\ref{fig:width} for an illustration in one dimension. The next theorem shows that $W_n$ scales as $n^{1/\beta-1}$ as $n \to \infty$.

        \begin{figure}
            \centering
            \includegraphics[width=0.6\linewidth]{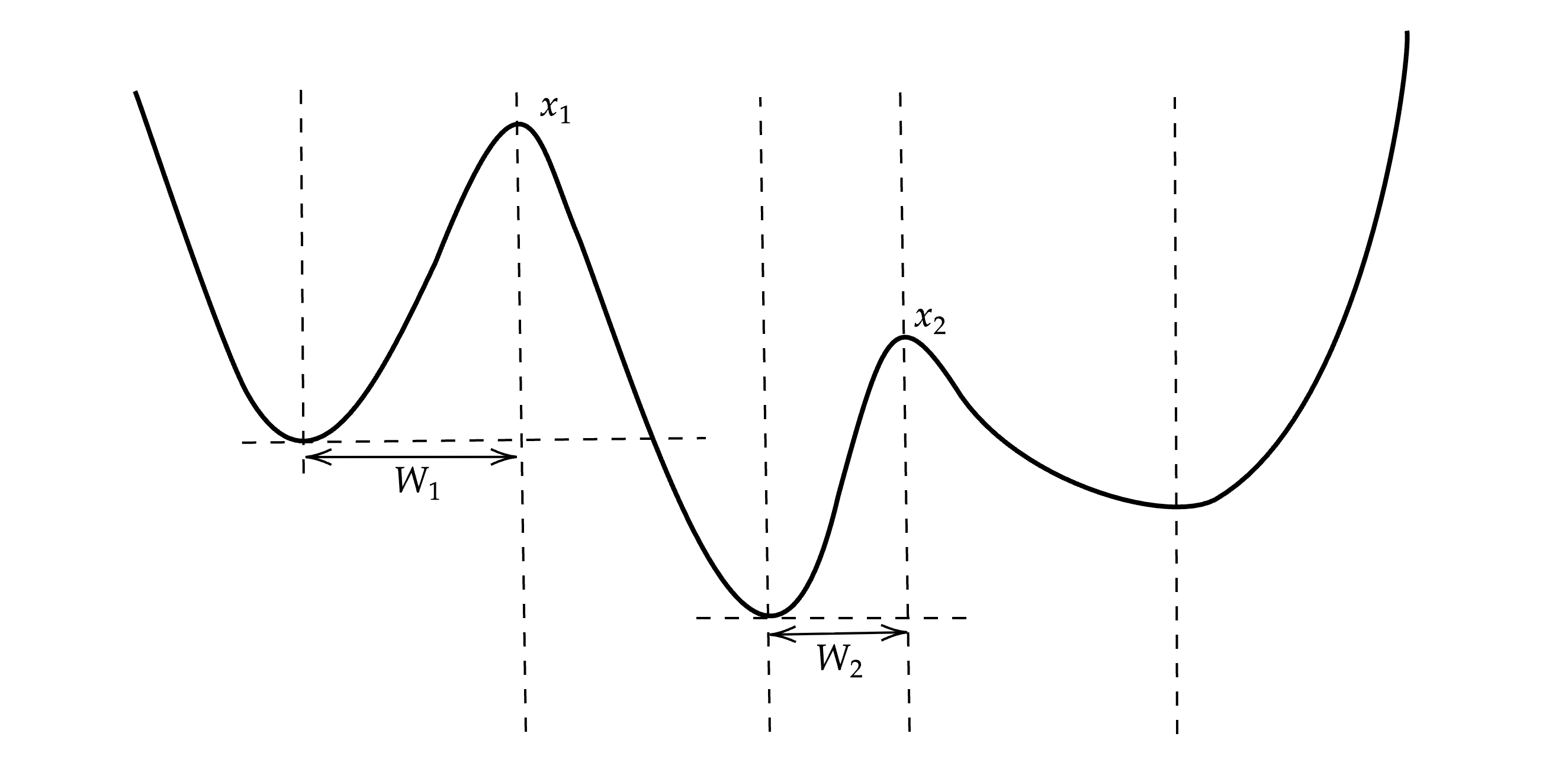}
            \caption{The widths of local maxima at $x_1$ and $x_2$ for some arbitrary function $f$.}
            \label{fig:width}
        \end{figure}
        
    \begin{theorem}
    \label{theo:width}
        Suppose Assumption \ref{assum:density} holds with $0<\beta\le 1$. Let $k \ge 0$ be fixed and $n$ be such that Theorem~\ref{theo:maxima} is in force. Let $W_n$ be the width of the micromode as defined by \eqref{eq:width}. Then, on the event $E_n$,
        \[
            \left(\frac{G_k+c_{\beta}}{2}\right)n^{1/\beta -1} - 2\sqrt{\nu} < W_n < 4G_kn^{1/\beta -1} + \sqrt{\nu},
        \]
        where $c_{\beta}$ is as defined in \eqref{eq:cbeta}.
    \end{theorem}

    Theorem~\ref{theo:width}, together with Theorem~\ref{theo:maxima}, characterizes the micromodes occurring due to extreme observations. In particular, Theorem~\ref{theo:width} gives an estimate of the size of a micromode when it exists. From a computational perspective, these results imply the existence of an energy barrier in the posterior distribution which an MCMC algorithm needs to overcome to leave a micromode. In the next section, we study the time required by ZZP to overcome this energy barrier when initialized in one of the micromodes.

    \subsection{Exit times for ZZP to leave a micromode}
    \label{sec:ek-formula}
        For the remainder of this section, we restrict our attention to the one-dimensional case and particularly focus on the micromode associated with $Y_{(n)}$ i.e. the farthest data point. The reader may observe that even this seemingly simple setting demands a non-trivial analysis.

        A one-dimensional ZZP  $(X_t, V_t)_{t\ge 0}$ is a continuous-time Markov process defined on the augmented space of position and velocity $\mathcal{Z} = \mathbb{R} \times \{-1,+1\}$. Given a switching rate $\lambda\colon \mathcal{Z} \mapsto \RR^+$, its dynamics are determined by the extended Markov generator
        \begin{equation}
            \mathcal{L}f(x, v) = v\cdot \partial_xf(x, v) + \lambda(x, v)\left\{f(x, -v) - f(x, v)\right\}, \quad (x,v) \in \mathcal{Z},
        \end{equation}       
    for functions $f$ in the core of the generator \citep[see e.g.][]{davis2018markov}. Loosely speaking, $X_t$ moves with a constant velocity $V_t$ and, at random times with rate $\lambda(X_t, V_t)$, it switches its velocity values between $+1$ and $-1$. For each $n$, let $Z^n_t = (X^n_t, V^n_t)$ be the one-dimensional ZZP with switching rate $\lambda_n$. If $\lambda_n(x,1) - \lambda_n(x,-1) = -\partial\log\pi^{(n)}(x) = (\nu + 1)S_n(x)$ for all $x$, then, the marginal distribution of $X^n_t$ is invariant for $\pi^{(n)}$ \citep[see][for details]{bierkens_zig-zag_2019}. We consider the following two choices of switching rates:
        \begin{equation}
        \label{eq:switching_rates}
            \begin{gathered}
                \lambda^{\text{can}}_n(x, v) = \left(v\partial \log \pi^{(n)}(x)\right)_{+} = (\nu + 1)\left(v \sum_{j=1}^n S(x, Y_j)\right)_{+}, \\
                \lambda^{\text{ss}}_n(x, v) = (\nu + 1)\sum_{j=1}^n\left(v S(x, Y_j)\right)_{+}.
            \end{gathered}
        \end{equation}
        The choice of $\lambda^{\text{can}}_n$ and $\lambda^{\text{ss}}_n$ corresponds to canonical ZZP and subsampling ZZP respectively of \cite{bierkens_zig-zag_2019}. Although the canonical ZZP can mix better than the subsampling ZZP \citep{andrieu2021peskun}, the subsampling ZZP has the advantage of reducing the implementation cost per iteration by a factor $n$. This is since $\lambda^{\text{ss}}_n(x,v)$ can be unbiasedly estimated by randomly selecting a single observation, say $Y_i$ and then computing $n(\nu + 1)\left(v S(x, Y_i)\right)_{+}$, thus  often improving the overall computational efficiency compared to canonical ZZP \citep[see][for details]{agrawal_large_2024}. We will drop the superscript when the choice of $\lambda_n$ is clear from the context.

        For large $n$, let $x_n^{+}$ be the point of local maximum near $Y_{(n)}$ given by Theorem~\ref{theo:maxima} on the event $E_n$. Let $W_n$ denote the width of the associated micromode defined by \eqref{eq:width}. Suppose the initial condition $X^n_0 = x_n^{+}$ for the ZZP targeting $\pi^{(n)}$. The time taken by the process to exit the micromode is then the first time $t \ge 0$ such that $|X^n_t - x_n^{+}| > W_n$. Without loss of generality, suppose $Y_{(n)} > 0$ (the arguments for $Y_{(n)} < 0$ will follow similarly). From \eqref{eq:empirical_score} and the positive-definiteness of $S'_n$ in the $\sqrt{\nu}$-neighborhood of $Y_{(n)}$ (Lemma~\ref{lemma:pd_property}), it follows that $S_n(x) > 0$ for all $x \ge x_{n}^{+}$ i.e. $\pi^{(n)}$ is monotonically decreasing in the half-line $[x_n^{+}, \infty)$. In this case, the exit on the right is irrelevant since the process will always return back to the micromode with probability $1$. 
        Thus, we will say that the process has exited the micromode when it crosses the width $W_n$ on the left for the first time. More precisely, we define the exit time $\tau_n$ for the ZZP $Z^n_t = (X^n_t, V^n_t)$ by 
        \begin{equation}
        \label{eq:tau_n}
            \tau_n = \inf\left\{t \ge 0: x_n^{+} - X^n_t \ge W_n  \mid (X^n_0, V^n_0) = (x_n^{+}, -1)\right\}.
        \end{equation}
        To study $\tau_n$ precisely, both, the point $x_n^{+}$ and the width $W_n$ must be exactly known. This is a formidable task even for small $n$, see e.g. \cite{ferguson_maximum_1978}. However, it turns out the estimate of the width  in Theorem~\ref{theo:width} is sufficient to estimate $\tau_n$ up to an order of magnitude.
        \begin{theorem}
        \label{theo:arrhenius_law}
            Suppose Assumption~\ref{assum:density} holds with $0 < \beta \le 1$. For all $n \ge 3$, let $Z^n_t$ be a Zig-Zag process with switching rates given by \eqref{eq:switching_rates}. Let $\tau_n$ be as defined in \eqref{eq:tau_n}. Then, for both $\lambda_{n}^{\text{can}}$ and $\lambda_{n}^{\text{ss}}$, as $n \to \infty$
            \[
                \mathbb{E}(\tau_n \mid E_n) = O_{p}\left(n^{(1/\beta - 1)(\nu + 1)}\right),
            \]
            with $E_n$ as in \eqref{eq:en_sets}.
        \end{theorem}
Therefore, subsampling ZZP and canonical ZZP take approximately the same time to leave the micromode, rendering subsampling ZZP an attractive choice given its favorable computational cost.

\subsection{Essential micromodes and model misspecification}
\label{sec:essential}
    It is common knowledge that a lighter tailed model for heavy tailed datasets i.e., $\nu > \beta$, leads to larger weights of the likelihoods assigned to the outliers. This, as a consequence, leads to a larger time spent by MCMC methods near the corresponding  micromodes of the posterior distribution. Theorem \ref{theo:arrhenius_law} captures this relationship in a rigorous manner for the Bayesian model in \eqref{eq: linear model} and ZZP.

   Consider now a ZZP initialized at $Y_{(n)}$ in its position component. If it does not get stuck in a micromode, it still requires $O_p(|Y_{(n)}|)$ amount of time to reach the posterior concentration, since the process moves with unit speed $\pm 1$. By Proposition~\ref{propo:evt},  $Y_{(n)} = O_{p}(n^{1/\beta})$, so this travel time  is of the same order and diverges as $n \to \infty$. We therefore rescale the time of the process by a factor $n^{1/\beta}$. If the exit time of the original process is  $\tau_n \ll n^{1/\beta}$, then, after rescaling, this time becomes negligible compared to the total time needed to reach the bulk of the distribution. Conversely, if $\tau_n \gg n^{1/\beta}$, the process remains stuck near $Y_{(n)}$ for an unreasonable time even after rescaling. Naturally, the latter is an undesirable situation. In this case, we say that the micromode is an {\it essential micromode}. More precisely, a micromode is an essential micromode if,
    \begin{equation}
    \label{eq:essential_micromode}
         \frac{\mathbb{E}\left(\tau_n \mid E_n\right)}{|Y_{(n)}|} \to \infty
    \end{equation}
    as $n \to \infty$. By Theorem \ref{theo:arrhenius_law} and Proposition \ref{propo:evt}, essential micromodes occur when 
    \[
        (1/\beta - 1)(\nu +1) > 1/\beta \iff \nu > \frac{\beta}{1 - \beta},
    \]
    and thus, it depends on a combination of the assumed model and the true data distribution $P$.
        \begin{figure}
        \centering
        \includegraphics[width = 0.49\textwidth]{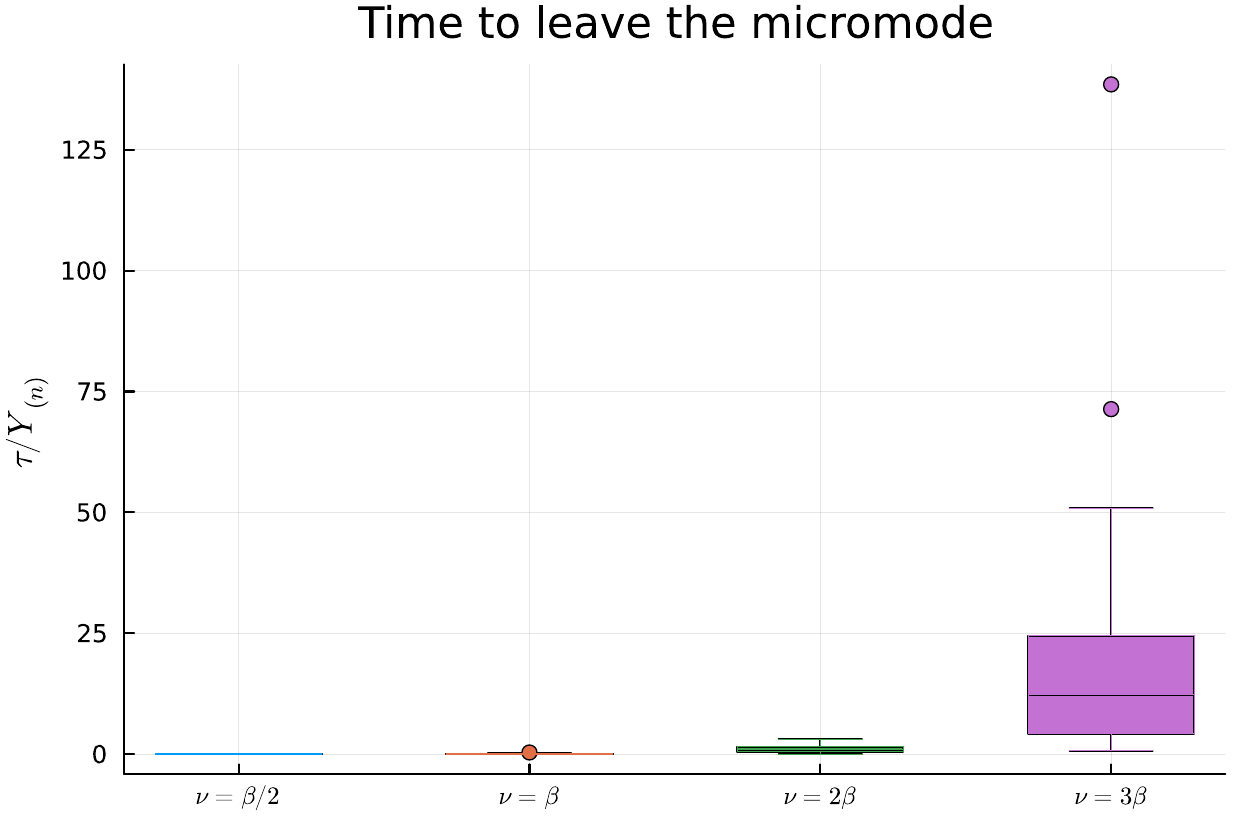}
        \includegraphics[width = 0.49\textwidth]{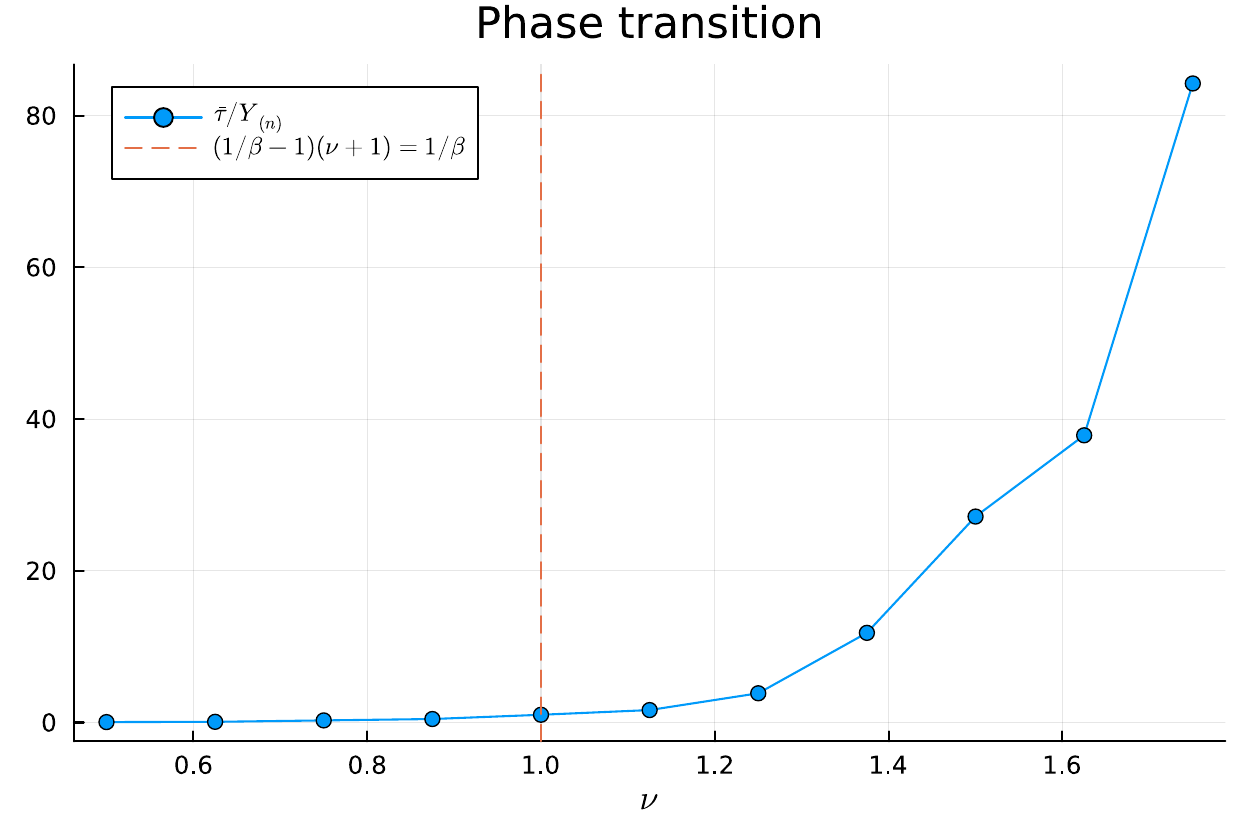}  
    \caption{Time to leave the micromode near $Y_{(n)}$ normalized by $|Y_{(n)}|$ for canonical ZZP targeting $\pi_\nu$ and started in the micromode near $Y_{(n)}$, for different values of $\nu$ ($x$-axis). Here, $n = 3000$ and $\beta = 1/2$. Left: each boxplot was computed with 40 independent trajectories. Right: Average time to leave the micromode computed over 40 independent trajectories targeting $\pi_\nu$, for different $\nu$ ($x$-axis). Dashed line at the phase transition $\beta/(1 - \beta) = 1$.}
    \label{fig:phase transition}
    \end{figure}
    
    Given the above, observe that a well-specified model i.e. $\beta = \nu$ will never give rise to an essential micromode since $0 < \beta \le 1$ and the same is true for (i) overdispersed models i.e. $\nu < \beta$ and (ii) when $\beta = 1$. However, when 
    $0 < \beta < 1$, there is a phase transition for which lighter tailed models beyond the threshold $\beta/(1 - \beta)$ give rise to essential micromodes that trap the process for long durations of time.
    This phase transition is illustrated in Figure~\ref{fig:phase transition} for $\beta = 1/2$. As can be seen in the figure, $\tau_n/ |Y_{(n)}|$ starts diverging for $\nu > 1$.

\section{Empirical score approximation}
\label{sec:score_approximation}
    In this section, we formalize the arguments illustrated in Section~\ref{sec:intro}. Let $Y_1, \dots, Y_n \sim P$ be independently distributed random vectors in $\mathbb{R}^d$ and suppose $P$ satisfies Assumption \ref{assum:density} with $0 < \beta \le 1$.

    The main technical result of this section is Lemma \ref{lemma:score_bound}, which shows that the empirical score can be approximated by a data-independent function in remote locations of the space with high precision. 

    \begin{lemma}
    \label{lemma:score_bound}
       Define, for each $n \ge 4$,
        \begin{equation}
        \label{eq:boldS}
            \mathbb{S}_n(z) = \sup_{x \colon |x-z| \le 2|z|/n}\frac{2|z|}{n}\left|S_n(x) - \frac{nx}{|x|^2}\right|.
        \end{equation}
         Let $0 < \epsilon <  1$ be arbitrary. There exists a $T_0(\epsilon)$ such that for all $n \ge T_0(\epsilon)$ and $|z| > 2n\sqrt{\nu}$,
        \begin{equation}
        \label{eq:boldS_bound}
            \mathbb{S}_n(z) \le 48\nu n^{2\epsilon-1} + \frac{|z|^{-\beta}}{n^{\epsilon(d-1)}}L'(|z|) + g(|z|) + R_n(|z|)
        \end{equation}
        where $L'(|z|) \in RV_{0}$, $g(|z|)  \in RV_{-\beta}$ and $R_n$ is a random variable depending on $Y_1, \dots, Y_n$ and $|z|$ such that for all $t > 0$,
        \begin{equation}
        \label{eq:Rn_bound}
            \PP\left(R_n(|z|) > t\right) \le 5^d\exp\left(-\frac{n^{1-2\epsilon}t^2}{512}\right) + C_d\frac{5^d|z|^{-\beta}}{2^dn^{d\epsilon-1}}L(|z|).
        \end{equation}
        Here $C_d$ is a dimension-dependent constant and $L(t) = t^{\beta+d}h(t)$ for $h$ in Assumption \ref{assum:density}. 
    \end{lemma}

    The result in Lemma~\ref{lemma:score_bound} is essentially an error bound on the law of large numbers for the empirical score $S_n$. For $|z| > 2n\sqrt{\nu}$ and arbitrary $\epsilon > 0$, with high probability, the bulk of the dataset lie atleast $|z|/n^{2\epsilon}$ distance away from the ball $B_{2|z|/n}(z)$. For all such datapoints, the score function remains uniformly bounded by $n^{1+2\epsilon}/|z|$ and the standard concentration inequalities apply. On the other hand, the total score contribution by the part of the dataset closer to or inside the ball is controlled trivially by the number of such datapoints. We prove Lemma~\ref{lemma:score_bound} by combining these observations with a uniform approximation of the expected score, i.e., $\EE(S_n(x))$ by the function $x \mapsto nx/|x|^2$. See Appendix \ref{sec:score_approximation_proof} for details.

    \begin{remark}
        Let $(z_n)_{n=1}^{\infty}$ be a sequence with $|z_n| > 2n\sqrt{\nu}$.  Lemma \ref{lemma:score_bound} gives an error bound for the approximation of $S_n$ in the neighborhood of $z_n$. The approximation becomes sharper as $|z_n|$ goes to $\infty$. If $\mathbb{S}_n(z_n)$ goes to $0$ as $n \to \infty$, this means that the empirical score $S_n$ is strongly approximated by $nx/|x|^2$ in a $(|z_n|/n)-$neighborhood of $z_n$. Suppose $|z_n| \sim  n^{\alpha}$ for some $\alpha > 0$ such that $|z_n| > 2n\sqrt{\nu}$. By choosing $\epsilon < 1/2$, it follows from \eqref{eq:Rn_bound} that the right-hand side in \eqref{eq:boldS_bound} goes to $0$ for any $\alpha$ satisfying $\alpha\beta > (1- d/2)$. And so, if a data-point is artificially placed on $z_n$, it would be a potential micromode depending on the configuration of the rest of the dataset. If this data-point is sufficiently isolated from the rest of the dataset, this potential micromode will become an actual micromode.
    \end{remark}

    \begin{remark}
    \label{remark:Zn}
        The above discussion extends to any random sequence  $(Z_n)_{n=1}^{\infty}$ defined independently of the dataset $\{Y_1, Y_2, \cdots, Y_n\}$. Suppose $|Z_n| \sim n^{\alpha}$ for $\alpha > (1- d/2)/\beta$. Then, a simple conditioning argument together with Lemma \ref{lemma:score_bound} implies that for all $\delta < 1/2$,
        \[
            \PP\left(\mathbb{S}_n(Z_n) > n^{1-\alpha - \delta} \mid |Z_n| > 2n\sqrt{\nu}\right) \to 0
        \]
        as $n \to \infty$. Note that $\alpha$ must satisfy $\alpha \ge 1$ otherwise $\PP(|Z_n| > 2n\sqrt{\nu}) \to 0$ and the result is vacuous.
    \end{remark}

    Let $k \in \mathbb{N}$ be fixed. Then, by Proposition \ref{propo:evt}, $|Y_{(n-k)}| \sim n^{1/\beta}$ and $1/\beta > (1-d/2)/\beta$ for all $\beta$ and $d$. Lemma \ref{lemma:score_bound} then suggests that, around $Y_{(n-k)}$, the score contribution from the rest of $(n-1)$ observations can be approximated by $(n-1)x/|x|^2$. Motivated by this, we define an approximate score function in the neighborhood of $Y_{(n-k)}$, the $(n-k)-$th order statistic. Define $\mathcal{F}_n = \{x : |x-Y_{(n-k)}| < 2|Y_{(n-k)}|/n\}$ and
    \begin{equation}
    \label{eq:snhat}
        \hat{S}_n(x) = S_n(x)\bs{1}_{\mathcal{F}^c_n} +\left(\frac{(n-1)x}{|x|^2} + S(x, Y_{(n-k)})\right)\bs{1}_{\mathcal{F}_n}, \quad x \in \mathbb{R}^d.
    \end{equation}

    We now state the main result of this section.

    \begin{theorem}
    \label{theo:score_bound}
        Suppose $0 < \beta \le 1$ and fix $k\ge0$. For $n \ge k$, let the event $A_n = \{|Y_{(n-k)}| > 2n\sqrt{\nu}\}$. Then, for all $\delta < 1/2$, as $n \to \infty$,
        \begin{equation}
        \label{eq:result}
            \PP\left(\sup_{x \in \mathbb{R}^{d}}\left|S_n(x) - \hat{S}_n(x)\right| > n^{1-1/\beta - \delta} \mid A_n \right) \longrightarrow 0. 
        \end{equation}
    \end{theorem}

    \begin{remark}
    \label{remark:rejection}
        Note that $Y_{(n-k)}$ is not independent of the rest of the dataset. And so, Theorem~\ref{theo:score_bound} does not directly follow from Lemma~\ref{lemma:score_bound} as described in Remark~\ref{remark:Zn}. However, it is possible to use the procedure described in Remark~\ref{remark:Zn} by defining $Z_n$ to have the same distribution as $Y_{(n-k)}$ and conditioning on the event that $Z_n$ is indeed the $(n-k)-$th order statistic of the dataset $\{Y_{1}, \dots, Y_{n-1}, Z_n\}$. See Section~\ref{sec:order_stats} for details.
    \end{remark}

    \begin{remark}
        Since \eqref{eq:result} is essentially a law of large numbers result, it is further possible to improve the error rate to show that $|S_n - \hat{S}_n|_{\infty} \sim n^{1/2 - 1/\beta}$. However, we do not pursue this generalization in this paper.
    \end{remark}

    The usefulness of Theorem~\ref{theo:score_bound} is that it allows one to replace the empirical score $S_n$ by something much simpler and much more amenable to analysis. We use Theorem~\ref{theo:score_bound} in the next section to prove the existence of a local maximum in the neighborhood of $Y_{(n-k)}$ (Theorem~\ref{theo:maxima}) and to measure its width (Theorem~\ref{theo:width}). 
    These are the two main ingredients used to derive the exit time of a micromode for ZZP in Theorem~\ref{theo:arrhenius_law}, whose proof is presented in Section~\ref{sec:exit_time_body}.

\section{Existence and width of a micromode}
\label{sec:micromode_technical} 
   Here we will rigorously prove Theorem~\ref{theo:maxima} and Theorem~\ref{theo:width}. Note that since $\pi^{(n)}$ is continuously differentiable, there exists a local maximum on the closed ball $\overline{B_{\sqrt{\nu}}(Y_{(n-k)})}$ for any $k \ge 0$. We will show that, as $n \to \infty$, such a point is unique and lies at a distance of order $n^{1 - 1/\beta}$ from $Y_{(n-k)}$. Moreover, we will show the width \eqref{eq:width} of the associated micromode is of the order $n^{1/\beta-1}$. We first state a generic result that will be used to show that there exists exactly one local maximum of $\pi^{(n)}$ in a given neighborhood. 

    \begin{lemma}
    \label{lemma:unique_min}
        Let $U:\mathbb{R}^d \to \mathbb{R}$ be twice-continuously differentiable in a neighborhood of the closed ball $\overline{B_r(y)}$ for some arbitrary $r > 0$ and $y \in \mathbb{R}^d$. Suppose
        \begin{enumerate}
            \item for all $x \in \partial B_r(y)$, $\nabla U(x)^T(x - y) > 0$.
            \item for all $x \in \overline{B_{r}(y)}$, $\nabla^2 U(x) \succ 0$.
        \end{enumerate} 
        Then, there exists a unique $x^{*} \in B_r(y)$ such that 
        \[
            U(x^*) = \min_{x \in \overline{B_r(y)}}U(x)
        \]
        and moreover, $\nabla U(x^*) = 0$. 
    \end{lemma}

    \begin{proof}
        Since $U$ is continuous in a neighborhood of the closed ball $\overline{B_r(y)}$, it attains a minimum on $\overline{B_{r}(y)}$. Let $x^* \in \overline{B_{r}(y)}$ be a minimizer of $U$. Suppose for the sake of contradiction that $x^* \in \partial B_{r}(y)$. Let $t \in (0, 1)$. Then, by assumption,
        \[
            U(x^* + t(y- x^*)) \ge  U(x^*)
        \]
        since $x^* + t(y - x^*) = (1-t)x^* + ty \in \overline{B_{r}(y)}$. Dividing by $t$ and taking $t \downarrow 0$ gives,
        \[
            \nabla U(x^*)^T(y - x^*) \ge 0 \iff \nabla U(x^*)^T(x^* - y) \le 0
        \]
        which is a contradiction. So $x^*$ must lie inside the ball $B_r(y)$. Since $x^*$ is an unconstrained minimizer of $U$ in the open ball $B_r(y)$, it must satisfy the first-order condition $\nabla U(x^*) = 0$. The uniqueness follows from the convexity assumption that $\nabla^2 U(x) \succ 0$ for all $x \in \overline{B_{r}(y)}$.
    \end{proof}

    Let $G_{k}$ be as in the statement of Proposition~\ref{propo:evt}. Let $0 < \epsilon, \delta < 1/2$ be arbitrary. For $n \ge 3$, define the events,
    \begin{equation}
    \label{eq:en_sets}
        \begin{gathered}
             A := \left\{\omega \in \Omega: G_k(\omega) > c_{\beta}\right\}, \quad  A'_n := \left\{\min_{j \neq n-k}|Y_{(j)} - Y_{(n-k)}| > |Y_{(n-k)}|/2n^{\epsilon}\right\}, \\   
            B_n := \left\{\sup_{x}|S_n(x) - \hat{S}_n(x)| < n^{1-1/\beta - \delta}\right\}, \quad 
            E_n := A'_n \cap B_n \cap A,  
        \end{gathered}      
    \end{equation}
    where $\hat{S}_n$ is as defined in \eqref{eq:snhat} and $c_{\beta}$ is given by \eqref{eq:cbeta}. To prove Theorem~\ref{theo:maxima}, we will verify the two conditions in Lemma~\ref{lemma:unique_min} for $U(x) = (-\log\pi^{(n)})(x)$. First, the strong convexity condition. Recall that $S'_n(x)$ is the Jacobian matrix corresponding to $S_n(x)$, i.e. $S'_n(x) = \sum_{j=1}^n S'(x, Y_j)$ where
    \[
        S'(x, y) := \frac{1}{\nu + |x - y|^2}I_d - \frac{2}{(\nu + |x -y|^2)^2}(x - y)(x - y)^T.
    \]

    \begin{lemma}
    \label{lemma:pd_property}
        Let $0 \le r < \sqrt{\nu}$ be arbitrary. There exists an $N$, depending on $r$, large enough such that for all $n > N$, on the event $E_n$, $S'_n(x) \succ 0 \ \forall \ x \in \overline{B_{r}(Y_{(n-k)})}$.
    \end{lemma}

    \begin{proof}
        For arbitrary $a \in \mathbb{R}^d$ with $|a| = 1$, consider
        \begin{align*}
            a^TS'(x,y)a &= \frac{1}{\nu + |x - y|^2} - \frac{2|a\cdot (x - y)|^2}{(\nu + |x -y|^2)^2} \ge  \frac{(\nu - |x-y|^2)}{(\nu + |x-y|^2)^2}.
        \end{align*}
    
        On $E_n \subseteq A'_n$, $\min_{j\neq n-k} |Y_{(j)} - Y_{(n-k)}| > |Y_{(n-k)}|/2n^{\epsilon}$. Thus for all $x \in \overline{B_{r}(Y_{(n-k)})}$, it follows that on the event $A'_n$, $\min_{j \neq n-k}|x - Y_{(j)}| \ge |Y_{(n-k)}|/2n^{\epsilon} - \sqrt{\nu} =: \delta_n$. Since $|Y_{(n-k)}|/2n^{\epsilon}$ increases to infinity as $n$ becomes large, $\delta_n \to \infty$ as $n \to \infty$. Note that, the function $s \mapsto (\nu - s^2)/(\nu + s^2)^2$ is increasing on $(\sqrt{3\nu}, \infty)$. Thus, for big enough $n$ such that $\delta_n > \sqrt{3\nu}$, we get for all $j \neq (n-k)$,
        \[
            a^TS'(x, Y_{(j)})a \ge \frac{(\nu - \delta_n^2)}{(\nu + \delta_n^2)^2}, \quad x \in \overline{B_{r}(Y_{(n-k)})}.
        \]
        On the other hand for $j = (n-k)$, 
        \[
            a^TS'(x, Y_{(n-k)})a \ge \frac{\nu - r^2}{(\nu + r^2)^2}.
        \]
        Combining the two, we get, on the event $A'_n$ and for all $x \in \overline{B_{r}(Y_{(n-k)})}$,
        \[
            a^TS'_n(x)a \ge \frac{(n-1)(\nu - \delta_n^2)}{(\nu + \delta_n^2)^2} + \frac{\nu - r^2}{(\nu + r^2)^2} \to \frac{\nu - r^2}{(\nu + r^2)^2}
        \]
        as $n \to \infty$. To see this, observe that the first term scales as $ -n/\delta_n^2$ where $\delta_n \sim |Y_{(n-k)}|/n^{\epsilon} \sim n^{1/\beta - \epsilon}$ by Proposition \ref{propo:evt}. And so, $n/\delta_n^2 \to 0$ since $0 < \beta \le 1$ and $0 < \epsilon < 1/2$. Since $r < \sqrt{\nu}$, the above implies that there exists an $N$ such that for all $n \ge N$, on the event $A'_n$, we have $a^TS'_n(x)a > 0$ for all $x \in \overline{B_r(Y_{(n-k)})}$ and all $a$ with $|a| = 1$. This concludes the proof.
    \end{proof}

    The above lemma implies that if a local maximum of $\pi^{(n)}$ exists in the ball $B_{\sqrt{\nu}}(Y_{(n-k)})$, then it is unique and satisfies $S_n(x) = 0$. It remains to show that the local maximum exists close to $Y_{(n-k)}$. On the event $A$, let $m_1, m_2$ be such that $0 \le c_{\beta} < m_1(\omega) < m_2(\omega) < G_k(\omega)$ for each $\omega \in A$. Define,
    \begin{equation}
    \label{eq:x_n_dash}
        d_n := \frac{n^{1-1/\beta}}{m_1}; \quad d_n^{\pm} := \frac{1\mp\sqrt{1-4d_n^2\nu}}{2d_n}; \quad n \ge 3.
    \end{equation}

    \begin{figure}
        \centering
        \includegraphics[scale=0.38]{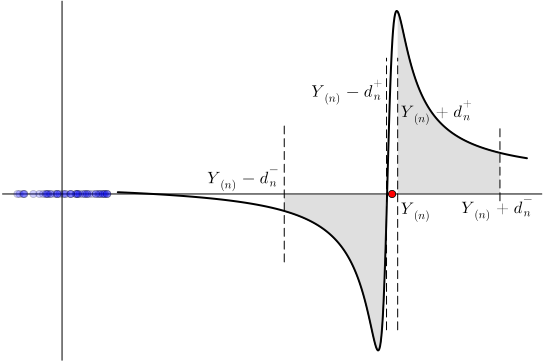}
        \includegraphics[scale=0.35]{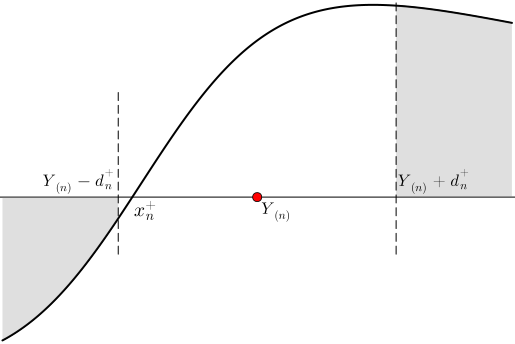}
        \caption{A diagram showing $S_n(x)$ and the construction of $d_n^{+}$ and $d_n^{-}$ in one dimension. The lengths $d_n^{\pm}$ are chosen such that $vS_n(Y_{(n)} + vt) > 0$ for all $t \in [d_n^{+}, d_n^{-}]$ and $v = \pm 1$. The right panel zooms in on the region where $S_n$ crosses the $x-$axis.}
        \label{fig:dn_figure}
    \end{figure}
    Note that $d_n^{\pm}$ are not constants but depend on the value of $G_k$ through $m_1$. From \eqref{eq:x_n_dash}, observe that $d_n^{\pm}$ are real if any only if $4d^2_n\nu < 1$. For $\beta = 1$, $d_n = 1/m_1 < 1/c_{\beta} = 1/(2\sqrt{\nu})$ by the choice of $m_1$ and definition of $c_{\beta}$. For $\beta < 1$, $d_n \to 0$ as $n \to \infty$. And so, there exists $N_1$ such that for all $n > N_1$, $d_n < 1/(2\sqrt{\nu})$ and $d_n^{\pm}$ are real. Additionally, on the event $A$, $n^{-1/\beta}|Y_{(n-k)}| \to G_k > m_2$. And so, there exists $N_2$ such that for all $n > N_2$, $n^{-1/\beta}|Y_{(n-k)}| > m_2$. This implies for all $n > \max\{N_1, N_2\}$,
    \[
        d_n^{+} \le 2d_n\nu < \sqrt{\nu} < \frac{\sqrt{\nu}}{2d_n} < d_n^{-} < \frac{1}{d_n} < \frac{m_2}{n^{1-1/\beta}} < \frac{|Y_{(n-k)}|}{n}.
    \]
    See Figure~\ref{fig:dn_figure} for an illustration in one dimension.
    
    \begin{lemma}
    \label{lemma:outward_grad}
       For large enough $n$, on the event $E_n$,
       \[
            \inf \left\{v^T S_n(Y_{(n-k)} + vt): t \in [d_n^{+}, d_n^{-}], v \in \mathbb{S}^{d-1} \right\} > 0.
       \]
    \end{lemma}
    \begin{proof}
        Suppose $n > \max\{N_1, N_2\}$, then $d_n^{-} \le |Y_{(n-k)}|/n$ and we have for all $v \in \mathbb{S}^{d-1}$ and $t \in [d_n^{+}, d_n^{-}]$,
        \begin{align*}
            v^T S_n(Y_{(n-k)} + vt) &> v^T \hat{S}_n(Y_{(n-k)} + vt) - n^{1-1/\beta-\delta} \\
            &= \frac{(n-1)}{|Y_{(n-k)} + vt|^2} v^T(Y_{(n-k)} + vt) + \frac{t}{\nu + t^2} - n^{1-1/\beta -\delta}.
        \end{align*}
        on the event $E_n \subseteq B_n$. Let $t \in [d_n^{+}, d_n^{-}]$ be fixed. For arbitrary $v \in \mathbb{S}^{d-1}$, let $\phi = v^TY_{(n-k)}$. Then, $\phi \in [-|Y_{(n-k)}|, |Y_{(n-k)}|]$. Consider the function,
        \[
            f(\phi) = (n-1)\frac{\phi + t}{|Y_{(n-k)}|^2 + t^2 + 2\phi t} \implies f'(\phi) = (n-1)\frac{|Y_{(n-k)}|^2 - t^2}{(|Y_{(n-k)}|^2 + t^2 + 2\phi t)^2}.
        \]
        Since $t < d_n^{-} < |Y_{(n-k)}|$, the function $f$ is monotonically increasing in $\phi$. Thus by setting $\phi = -|Y_{(n-k)}|$, we have,
        \[
            v^T S_n(Y_{(n-k)} + vt) \ge -\frac{n-1}{|Y_{(n-k)}| - t} + \frac{t}{\nu + t^2} - n^{1-1/\beta -\delta},
        \]
        for all $v \in \mathbb{S}^{d-1}$. The first term is decreasing in $t$. Moreover, by construction, $t/(\nu + t^2) > d_n$ for all $t \in [d_n^{+}, d_n^{-}]$. Thus,
        \[
            v^T S_n(Y_{(n-k)} + vt) \ge -\frac{n-1}{|Y_{(n-k)}| - d_{n}^{-}} + d_n - n^{1-1/\beta -\delta} \ge \frac{-n}{|Y_{(n-k)}|} + d_n - n^{-\delta}m_1d_n,
        \]
        since $d_n^{-} < |Y_{(n-k)}|/n $. But also, $|Y_{(n-k)}|/n > m_2/n^{1-1/\beta}$ which gives,
        \[
            v^T S_n(Y_{(n-k)} + vt) \ge n^{1-1/\beta}\left(\frac{-1}{m_2} +  \frac{1}{m_1} - n^{-\delta}\right) > 0
        \]
        for $n$ large enough.
    \end{proof}
  We are now ready to prove both Theorem~\ref{theo:maxima} and  Theorem~\ref{theo:width}.
    \begin{proof}[Proof of Theorem~\ref{theo:maxima}]
        Let $r = \sqrt{\nu}/2$ if $0 < \beta < 1$ and $r = 2\nu/m_1$ if $\beta = 1$. Then for $n > \max\{N_1, N_2\}$, $d_n^{+} < r < \sqrt{\nu}$ for any $\beta \le 1$ on the event $E_n$. Let $N > \max\{N_1, N_2\}$ be such that for all $n > N$, the statements of Lemma~\ref{lemma:pd_property} and Lemma~\ref{lemma:outward_grad} both hold. Then for $n > N$, on the event $E_n$, it follows that $S'_n(x) \succ 0$ in the closed ball $\overline{B_{r}(Y_{(n-k)})}$ and  $v^T(S_n(Y_{(n-k)} + vt)) > 0$ for all $v \in \mathbb{S}^{d-1}$ and $t \in [d_n^{+}, d_n^{-}]$. By Lemma~\ref{lemma:unique_min}, this implies that for all $n > N$, on the event $E_n$, there exists a unique $x_n^{+} \in B_{d_n^{+}}(Y_{(n-k)})$ such that $S_n(x_n^{+}) = 0$ and $S'_n(x_n^{+}) \succ 0$. In addition, since $d_n^{-} > \sqrt{\nu}$, this further means that there cannot be any other root of $S_n(x)$ on the annulus $\overline{B_{\sqrt{\nu}}(Y_{(n-k)})}\setminus B_{d_n^{+}}(Y_{(n-k)})$. Thus, there exists a unique local maximum $x_n^{+}$ in the $\sqrt{\nu}-$nbd. of $Y_{(n-k)}$. Further, by design,
        \[
            |Y_{(n-k)} - x_n^{+}| \le d_n^{+} < 2d_n\nu = \frac{2\nu}{m_1}n^{1-1/\beta}.
        \]
        Choosing $m_1 = (G_k + c_{\beta})/2$ gives the first part.

        For the second part of the theorem, consider the sequence $A_n = \{|Y_{(n-k)}| > 2n\sqrt{\nu}\}$. Observe that $A$ is eventually contained in $A_n$ i.e. $A \subset \liminf_{n \to \infty} A_n$. Thus, for large $n$, $E_n \subseteq A \subset A_n$ and $\PP(E_n \mid A_n) = \PP(E_n)/\PP(A_n)$. But also since $G_k$ is continuous, $\lim_{n \to \infty} \PP(A_n) = \PP(A)$. From Theorem~\ref{theo:score_bound}, it follows that $\PP(B_n | A_n) \to 1$ as $n \to \infty$. This gives $\PP(A \cap B_n) \to \PP(A)$ as $n \to \infty$. Finally, $\PP(A'_n) \to 1$ by Proposition~\ref{propo:evt}. This implies $\PP(A'_n \cap B_n \cap A) = \PP(E_n)\to \PP(A)$ and hence, $\PP(E_n \mid A_n) \to 1$.
    \end{proof}
    \begin{proof}[Proof of Theorem~\ref{theo:width}] 
        Define
        \[
            t_0 = \sup\{t > 0: |x_n^{+} + vt - Y_{(n-k)}| < d_n^{-} \ \forall\ v\in \mathbb{S}^{d-1}\}.
        \]
        Let $t < t_0$. If $|x_n^{+} + vt - Y_{(n-k)}| > d_n^{+}$, then $v^TS_n(x_n^{+} + vt) > 0$ by Lemma~\ref{lemma:outward_grad}. Suppose $|x_n^{+} + vt - Y_{(n-k)}| < d_n^{+} < \sqrt{\nu}$. A first-order Taylor's expansion gives for some $0 < r < t < \sqrt{\nu}$
        \[
            v^TS_n(x_n^{+} + vt) = v^TS_n(x_n^{+}) + v^TS'_n(x_n^{+} + rv)v = v^TS'_n(x_n^{+} + rv)v.
        \]
        Since $|x_n^{+} - Y_{(n-k)}| < d_n^{+}$, then the line joining $x_n^{+}$ and $x_n^{+} + vt$ also lies entirely inside the $d_n^{+}-$ball of $Y_{(n-k)}$. And so, the right-hand side above is positive by Lemma~\ref{lemma:pd_property} above. This implies that 
        \[
            W_n \ge t_0 \ge d_{n}^{-} - |x_n^{+} - Y_{(n-k)}| > d_n^{-} - d_n^{+} = \frac{\sqrt{1 - 4d_n^2\nu}}{d_n} > \frac{1}{d_n} - 2\sqrt{\nu}.
        \]
        The lower bound follows by putting $d_n$ from \eqref{eq:x_n_dash} with $m_1 = (G_k+c_{\beta})/2$. 

        For the upper bound, let $r = |Y_{(n-k)}|$, $v = -Y_{(n-k)}/r$ and $t = 2r/n $. Consider,
        \[
            v^T\hat{S}_n(Y_{(n-k)} + vt) = \frac{2r/n}{\nu + (2r/n)^2} - \frac{n-1}{r-2r/n} < \frac{-n}{r}\left(\frac{1}{2} + \frac{1}{n-2}\right).   
        \]
        And so, on the event $E_n$, $v^TS_n(Y_{(n-k)} + vt) \le v^T\hat{S}_n(Y_{(n-k)} + vt) + n^{1-1/\beta - \delta} < 0$ for $n$ large enough. This implies that,
        \[
            W_n < |Y_{(n-k)} + vt - x_n^{+}| \le |Y_{(n-k)} - x_n^{+}| + 2|Y_{(n-k)}|/n.
        \]
        Since $n^{-1/\beta}|Y_{(n-k)}| \to G_k > c_{\beta}$, there exists $N$ such that for $n > N$, $|Y_{(n-k)}|/n < 2G_k n^{1/\beta-1}$. And so, the result follows.
    \end{proof}

\section{Exit time from a micromode}
\label{sec:exit_time_body}
    In this section, we study the time taken by ZZP to exit a micromode in one dimension and prove Theorem~\ref{theo:arrhenius_law}. 

    For $n \ge 4$, let $Z^n_t = (X^n_t, V^n_t)$ be either canonical ZZP or ZZP with subsampling  with switching rates $\lambda_n$ defined by \eqref{eq:switching_rates}. Suppose $n$ is large enough such that on the event $E_n$, both Theorem~\ref{theo:maxima} and Theorem~\ref{theo:width} are in force. Given the initial condition $(X^n_0, V^n_0) = (x_n^{+}, -1)$, we are interested in the exit time of $Z^n_t$ defined in \eqref{eq:tau_n}. Define $x_n^{-} := x_n^{+} - W_n$. Then, 
    \[
        \tau_n = \inf\{t \ge 0: X^n_t \le x_n^{-} \mid (X^n_0, V^n_0) = (x_n^{+}, -1)\}.
    \]
    
    Following the approach of \cite{monmarche_piecewise_2016}, we study the exit times of $Z^n_t$ using standard renewal-time arguments. Towards that, define, 
    \begin{equation}
    \label{eq:defns}
        \begin{gathered}
            \eta_n = \inf\{t > 0: X^n_t \in \{x_n^{-}, x_n^{+}\}\}; \quad p_{\tau, n} = \PP(X^n_{\eta_n} = x_n^{-}). \\
            T_n(y, v) = \inf\{t > 0: (X^n_t, V^n_t) = (x_n^{+}, -1) \mid (X^n_0, V^n_0) = (y, v)\}
        \end{gathered}
    \end{equation}
    Starting at $(x_n^{+}, -1)$, $\eta_n$ is the first time the ZZP either exits the micromode or rises back to the local maximum $x_n^{+}$. By construction, $Z^n_{\eta_n} \in \{(x_n^{-}, -1), (x_n^{+}, +1)\}$ with probability $1$. Given the initial condition, if $Z^n_{\eta_n} = (x_n^{-}, -1)$, then $\tau_n = \eta_n$ and the process exits the micromode on the left. Conversely, if $Z^n_{\eta_n} = (x_n^{+}, +1)$, the ZZP takes another $T_n(x_n^{+}, +1)$ time to return to the initial condition $(x_n^{+}, -1)$ and begin the next excursion. Since the ZZP is ergodic \cite[see e.g., ][]{bierkens_ergodicity_2019} with $\PP-$probability $1$
    , it follows that $\PP(T_n(x_n^{+}, +1) < \infty) = 1$. We obtain the following recursion for the exit time
    \[
        \tau_n = \begin{cases}
            \eta_n & \text{with prob. } p_{\tau, n} \\
            \eta_n + T'_n(x_n^{+}, +1) + \tau'_n & \text{with prob. } 1 - p_{\tau, n}.
        \end{cases} 
    \]
    where $T'_n(x_n^{+}, +1)$ and $\tau'_n$ are iid copies of $T_n(x_n^{+}, +1)$ and $\tau_n$ respectively.
    So, on the event $E_n$,
    \begin{equation}
    \label{eq:expected_tau}
        \EE(\tau_n) = \frac{1}{p_{\tau, n}} \EE(\eta_n) + \frac{1 - p_{\tau, n}}{p_{\tau, n}}\,\EE(T_n(x_n^{+}, +1)).
    \end{equation}
    
    In \eqref{eq:expected_tau}, $1/p_{\tau, n}$ represents the expected number of excursions the ZZP makes to successfully exit the micromode on the left. Next, we will show that this is the leading order term. Consider the switching rates \eqref{eq:switching_rates}. Observe that, the rate
    \[
        \lambda_n(x, v) = (\nu + 1)(v \cdot S_n(x))_{+} + \gamma_n(x), \quad (x, v) \in \mathbb{R} \times \{-1, +1\},
    \]
   corresponds to the canonical ZZP rate for $\gamma_n \equiv 0$ and the subsampling ZZP rate for
    \begin{equation}
    \label{eq:excess_rate}
        \gamma_n(x) = (\nu + 1)\left(\sum_{j=1}^n(S(x, Y_j))_{+} - \left(S_n(x)\right)_{+}\right), \quad x \in \mathbb{R}.
    \end{equation} 
     Recall from Section~\ref{sec:ek-formula} that $S_n(x) \ge 0$ for all $x \ge x_n^{+}$. Moreover from \eqref{eq:empirical_score}, $S_n(x) \ge 0$ for all $x\ge Y_{(n)}$. And so, it must be that $x_n^{+} < Y_{(n)}$. Also, by definition of $W_n$ in \eqref{eq:width}, it follows that $S_n(x) \le 0$ for all $x \in (x_n^{-}, x_n^{+})$. Finally, on the set $E_n$, $|Y_{(n)} - Y_{(j)}| > |Y_n|/2n^{\epsilon} > W_n$ for large $n$. Thus, $x_n^{-} > |Y_{(j)}|$ for all $j \le n-1$ so that $S(x, Y_{(j)}) > 0$. We get for large enough $n$,
    \begin{align}
    \label{eq:excess_cases}
        \frac{\gamma_n(x)}{\nu + 1} &= \sum_{j=1}^n(S(x, Y_j))_{+} - \left(S_n(x)\right)_{+} = \begin{cases}
            S_n(x) - S(x, Y_{(n)}), & x_n^{-} \le x < x_n^{+}, \\
            -S(x, Y_{(n)}), & x_n^{+} \le x < Y_{(n)}, \\
            0, & x \ge Y_{(n)}.
        \end{cases}
    \end{align}
    In any case, $\gamma_n(x)/(\nu + 1) \le S_n(x) - S(x, Y_{(n)})$ for $x \in (x_n^{-}, Y_{(n)})$. For $x \in (x_n^{-}, Y_{(n)})$, on the event $E_n$ in \eqref{eq:en_sets},
    \begin{align}
    \label{eq:gamma_bound}
        \frac{\gamma_n(x)}{\nu + 1} \le S_n(x) - S(x, Y_{(n)}) &\le \hat{S}_n(x) + n^{1-1/\beta-\delta} - S(x, Y_{(n)}) =\frac{n-1}{x} + n^{1-1/\beta-\delta}.
    \end{align}
    The first term on the right-hand side is decreasing in $x$ and we have, $x > x_n^{-} = Y_{(n)} - (Y_{(n)} - x_n^{+}) - W_n > Y_{(n)} - 2\sqrt{\nu} - 4G_kn^{1/\beta-1}$. But also from the proof of Theorem~\ref{theo:maxima}, $n$ is such that $|Y_{(n)}| > m_1n^{1/\beta}$. This gives,
    \[
        \frac{\gamma_n(x)}{\nu + 1} \le n^{1-1/\beta}\left(\frac{1}{m_1 - 2\sqrt{\nu}n^{-1/\beta} - 4G_kn^{-1}} - n^{-\delta}\right) \le \frac{n^{1-1/\beta}}{2m_1}
    \]
    for large $n$. Define $\overline{\gamma_n} := n^{1-1/\beta}/(G_k + c_{\beta})$. Then for large $n$, $\gamma_n(x) \le (\nu + 1)\overline{\gamma_n}$ for all $x \in (x_n^{-}, x_n^{+})$.
    Moreover, from Theorem~\ref{theo:width},
    \begin{align}
    \label{eq:w_times_gamma}
        \frac{1}{2} - \frac{2\sqrt{\nu} n^{1-1/\beta}}{G_k+c_{\beta}} \le W_n\overline{\gamma_n} \le \frac{4G_k}{G_k + c_{\beta}} +  \frac{\sqrt{\nu}n^{1-1/\beta}}{G_k + c_{\beta}} \le 4 +  \frac{\sqrt{\nu}n^{1-1/\beta}}{G_k + c_{\beta}}.
    \end{align}
    Let $p_n := \pi^{(n)}(x_n^{-})/\pi^{(n)}(x_n^{+})$. By Proposition 17 of \cite{monmarche_piecewise_2016},
        \[
           \frac{p_n}{1 + \overline{\gamma_n}W_n} \le  p_{\tau, n} \le p_n.
        \]
    Combining this with \eqref{eq:w_times_gamma} implies that $p_{\tau, n}$ is of the same order of magnitude as $p_n$. We have the following result.
    
    \begin{lemma}
    \label{lem:pn_magnitude}
        Suppose either $\gamma_n(x) \equiv 0$ or defined by \eqref{eq:excess_rate}. As $n \to \infty$,
        \[
            p_{\tau, n} \mid E_n = O_{p}\left(n^{(1-1/\beta)(\nu + 1)}\right).
        \]  
    \end{lemma}
    \begin{proof}
        From the preceding discussion, $p_n\left(1 - O_{p}(1)\right) \le p_{\tau,n}$ where $p_n = \frac{\pi^{(n)}(x_n^{-})}{\pi^{(n)}(x_n^{+})}$. 
        Define $\hat{p}_n := \exp\left((\nu + 1)\int_{x_n^{-}}^{x_n^{+}} \hat{S}_n(u)\ du\right)$. On the event $E_n$,
        \begin{align*}
            |\ln p_{n} - \ln \hat{p}_n| \le (\nu + 1)\int_{x_n^{-}}^{x_n^{+}} \left|S_n(u) - \hat{S}_n(u)\right|\ ds & \le (\nu + 1)n^{1-1/\beta -\delta}W_n.
        \end{align*}
        By Theorem~\ref{theo:width}, the right-hand side goes to $0$ as $n \to \infty$. Thus, on $E_n$, it follows that $|\ln p_{n} - \ln \hat{p}_n| = o_{p}(1)$. And so, $p_{\tau, n} = O_{p}(p_n) = O_p(\hat{p}_n)$. Now, solving the integral in the definition of $\hat{p}_n$ gives,
        \begin{align*}
            \hat{p}_n &= \left(\frac{x_n^{+}}{x_n^{-}}\right)^{(\nu + 1)(n-1)}\cdot\left(\frac{\nu + (x_n^{+} - Y_{(n)})^2}{\nu + (x_n^{-} - Y_{(n)})^2}\right)^{(\nu + 1)/2} \label{eq:pnhat}.
            % &=\left(1+ \frac{x_n^{+} - x_n^{-}}{x_n^{-}}\right)^{(\nu + 1)(n-1)}\cdot\left(\frac{\nu + (x_n^{+} - Y_{(n)})^2}{\nu + (x_n^{-} - Y_{(n)})^2}\right)^{(\nu + 1)/2}
            \numberthis 
        \end{align*}
        Since $Y_{(n)} - x_n^{+} < \sqrt{\nu} < Y_{(n)} - x_n^{-}$, and $x_n^{-} > Y_{n} - 2\sqrt{\nu} - 4G_kn^{1/\beta - 1}$ we have,
        \[
            \hat{p}_n \ge (\nu/2)^{(\nu + 1)/2}(Y_{(n)} - x_n^{-})^{-(\nu + 1)} \ge n^{(1-1/\beta)(\nu + 1)}\left(\frac{\sqrt{\nu/2}}{2\sqrt{\nu} + 4G_k}\right)^{(\nu + 1)}
        \]
        for large $n$. Similarly, from the proof of Theorem~\ref{theo:width},
        \[
            \frac{x_n^{-}}{x_n^{+}} \ge \frac{Y_{(n)} - 2\sqrt{\nu} - 4G_kn^{1/\beta-1}}{Y_{(n)}} \ge 1 - \frac{2n^{1-1/\beta}\sqrt{\nu} + 4G_k}{m_1n} \ge 1 - \frac{2\sqrt{\nu} + 4G_k}{m_1n}
        \]
        for big $n$. And so, 
        \[
            \hat{p}_n \le \left(1 - \frac{2\sqrt{\nu}+4G_k}{m_1n}\right)^{-(\nu + 1)(n-1)}(2\nu)^{(\nu + 1)/2}W_n^{-(\nu + 1)}.
        \]
        The result then follows by Theorem~\ref{theo:width}.
    \end{proof}

    Lemma~\ref{lem:pn_magnitude} implies that the expected number of excursions the ZZP makes before successfully exiting the micromode on the left is of the order $n^{(1/\beta - 1)(\nu + 1)}$. Additionally, the following result shows that the expected length of each excursion and the expected length of the excursion to the right, both remain $O_p(1)$ on the event $E_n$ as $n \to \infty$.

    \begin{lemma}
    \label{lemma:excursion_length}
        Suppose either $\gamma_n(x) \equiv 0$ or defined by \eqref{eq:excess_rate}. There exists constants $C_1, C_2$, and $C_3$ such that, 
        \begin{equation*}
            \begin{gathered}
                C_1 + O_{p}(n^{1-1/\beta}) \le \EE(T_{n}(x_n^{+}, +1) \mid E_n) \le C_2 + O_p(n^{1-1/\beta}),\\
                \EE(\eta_{n} \mid E_n) \le C_3 + O_p(n^{1-1/\beta}),
            \end{gathered}
        \end{equation*}
        as $n \to \infty$. 
    \end{lemma}

    We can now prove Theorem~\ref{theo:arrhenius_law}.

    \begin{proof}[Proof of Theorem~\ref{theo:arrhenius_law}]
        From \eqref{eq:expected_tau}, we have,
        \[
            \frac{1 - p_{\tau, n}}{p_{\tau,n}} \EE(T_n(x_n^{+}, +1)) \le \EE(\tau_n) \le \frac{1}{p_{\tau,n}} \EE(\eta_n) + \frac{1 - p_{\tau, n}} {p_{\tau,n}} \EE(T_n(x_n^{+}, +1)).
        \]
        The result then follows from Lemma~\ref{lem:pn_magnitude} and Lemma~\ref{lemma:excursion_length}.
    \end{proof}

\section{Extensions}
\label{sec:limitations}
    In this paper we choose extreme order statistics since they provide a natural and convenient sequence of isolated points. All the results can easily be extended to, e.g., intermediate order statistics (see \cite{de_haan_extreme_2006}), thanks to Lemma~\ref{lemma:score_bound}. On the computational side, we present the exit time analyses for ZZP only for the most extreme data point, $Y_{(n)}$. The same analysis can be extended to $Y_{(n-k)}$ for $k > 0$ after an additional (but tedious) control of right excursions. 

    Our analysis focuses on one-dimensional ZZP. Exit times for other popular gradient-based algorithms such as overdamped and underdamped Langevin diffusions have been studied by \cite{bovier2004metastability, lee2025eyring} for double-well potentials. Based on these results, we expect these algorithms to behave similarly in the setting considered in this paper. In particular, one could extend our micromode exit-time analysis to these processes in dimensions greater than 1. Although, we expect this extension to be technically more challenging since it typically relies on potential theory and spectral analysis \cite{bovier2016metastability}.

    Finally, we have restricted our attention in this paper to the simple model in \eqref{eq: linear model}. However, as demonstrated in Figure~\ref{fig:micromodes_and_target} in the introduction, the micromodes may also be commonly observed in more complicated models such as the linear model of the form $y_j = A_jx + \varepsilon_j$, $\varepsilon_j \sim t_1$. In this case, we expect the extreme observations to cause a ridge in the posterior surface along the hyperplane $y_{(n)} = A_{(n)}x$ as shown in Figure~\ref{fig:micromodes_and_target}. Our current proof technique via law of large numbers can still be applied here to approximate the empirical score in a tube around the hyperplane instead of a sphere around $y_{(n)}$. Similar analyses to that of Section~\ref{sec:micromode_technical} may then be carried out.

%%%%%%%%%%%%%%%%%%%%%%%%%%%%%%%%%%%%%%%%%%%%%%
%% Single Appendix:                         %%
%%%%%%%%%%%%%%%%%%%%%%%%%%%%%%%%%%%%%%%%%%%%%%
%\begin{appendix}
%\section*{???}%% if no title is needed, leave empty \section*{}.
%\end{appendix}
%%%%%%%%%%%%%%%%%%%%%%%%%%%%%%%%%%%%%%%%%%%%%%
%% Multiple Appendixes:                     %%
%%%%%%%%%%%%%%%%%%%%%%%%%%%%%%%%%%%%%%%%%%%%%%
\begin{appendix}
    \section{Maximum distance scaling}
\subsection{A result for order statistics}
\label{sec:order_stats}   
    Recall that $Y_1, \dots, Y_n \overset{iid}{\sim} P$ where $P$ satisfies Assumption \ref{assum:density}. Let $F_{R}$ denote the distribution function of the radial component of $P$ i.e. $F_R(r) = P(\{|y| \le r\})$. Then, under Assumption \ref{assum:density},
    \begin{equation}
    \label{eq:radial_df}
        F_{R}(r) := \int_{0}^rf_{R}(t)\ dt = \int_{0}^r \omega_{d-1} t^{d-1} h(t) dt, \quad r > 0,
    \end{equation}
    where $h$ is as in the Assumption \ref{assum:density} and $\omega_{d-1} = 2\pi^{d/2}/\Gamma(d/2)$ is the surface area of the unit sphere $\mathbb{S}^{d-1}$. We denote the density of the radial component by $f_{R}(\cdot)$.
    
    On an arbitrary probability space $(\Omega, \mathcal{F}, \mathbb{Q})$, define $\tilde{Y}_{n, 1}, \dots, \tilde{Y}_{n, n-1} \overset{iid}{\sim} P$. Define $Z_n$ to be a random vector independent of $\tilde{Y}_{n, 1}, \dots, \tilde{Y}_{n,n-1}$ with distribution $P_{(n-k)}$. Under Assumption \ref{assum:density}, we have
    \[
        \mathbb{Q}(Z_n \in A) = P_{(n-k)}(A) = \int_{A} \frac{f_{R_{(n-k)}}(|z|)}{\omega_{d-1}|z|^{d-1}}\ dz, \quad A \in \mathcal{B}(\mathbb{R}^d),
    \]
    where $f_{R_{(n-k)}}$ denote the density of the $k-$th largest radius of $n$ samples and is given by,
    \[
        f_{R_{(n-k)}}(r) = \frac{n!}{k!(n - k-1)!}[F_R(r)]^{n-k-1}[1 - F_{R}(r)]^{k}f_{R}(r), \quad r  > 0. 
    \]

    \begin{proposition}
    \label{propo:rejection}
        Let $g:(\mathbb{R}^d)^n \to \mathbb{R}$ be any function that is symmetric in its first $n-1$ arguments. Then, for any $A \in \mathcal{B}(\mathbb{R})$,
        \begin{equation}
        \label{eq:rejection_bound}
            \Prob(g(Y_{-(n-k)}, Y_{(n-k)}) \in A) \le 4\sqrt{k + 1} \cdot \mathbb{Q}(g(\tilde{Y}_{1, n-1}, \dots, \tilde{Y}_{n,n-1}, Z_n) \in A),
        \end{equation}
        where $Y_{-(n-k)}$ denotes the $(n-1)-$tuple $(Y_{1}, \dots, Y_n)$ after removing $Y_{(n-k)}$.
    \end{proposition}

    \begin{proof}
        Let $\hat{Y}_{n-k}$ be a random vector distributed according to $P_{(n-k)}$. Given $\hat{Y}_{n-k}$, define $\hat{Y}_1, \dots \hat{Y}_{n-k-1}$ to be distributed identically and independently according to $P$ with condition that the radial component is smaller than $\hat{Y}_{n-k}$ i.e.
        \[
            \PP(\hat{Y}_j \in B \mid \hat{Y}_{n-k}) = \frac{P(B\cap \{|Y| \le |\hat{Y}_{n-k}|\})}{P(\{|Y| \le |\hat{Y}_{n-k}|\})}, \quad B \in \mathcal{B}(\mathbb{R}^d), j = 1, \dots, n-k-1.
        \]
        Similarly, define $\hat{Y}_{n-k+1}, \dots, \hat{Y}_n$ to be identically and independently distributed according to $P$ with condition that the radial component is greater than $\hat{Y}_{n-k}$ i.e.
        \[
            \PP(\hat{Y}_j \in B \mid \hat{Y}_{n-k}) = \frac{P(B\cap \{|Y| > |\hat{Y}_{n-k}|\})}{P(\{|Y| > |\hat{Y}_{n-k}|\})}, \quad B \in \mathcal{B}(\mathbb{R}^d), j = n-k+1, \dots, n.
        \]
        Then, the distribution of $\{Y_{1}, \dots, Y_{n}\}$ is same as the distribution of $\{\hat{Y}_{1}, \dots, \hat{Y}_{n}\}$ up to permutations. In particular, $\hat{Y}_{(n-k)} = \hat{Y}_{n-k}$ almost surely, and we have,
        \[
            \Prob(g(Y_{-(n-k)}, Y_{(n-k)}) \in A) = \Prob(g(\hat{Y}_{-(n-k)}, \hat{Y}_{n-k}) \in A).
        \]

        Let $\{\tilde{Y}_{(n, 1)}, \dots,\tilde{Y}_{(n,n-1)}\}$ be the ordered statistics corresponding to $\{\tilde{Y}_{n,1}, \dots,  \tilde{Y}_{n,n-1}\}$. Consider the event,
        \[
            B_k = \{ |\tilde{Y}_{(n,n-k-1)}| \le |Z_n| \le |\tilde{Y}_{(n, n-k)}|\}.
        \]
        Then,
        \begin{align*}
            \Prob(g(\hat{Y}_{-(n-k)}, \hat{Y}_{n-k}) \in A) &= \mathbb{Q}(g(\tilde{Y}_{n,1}, \dots, \tilde{Y}_{n,n-1}, Z_n) \in A \mid B_k) \\
            &\le \frac{\mathbb{Q}(g(\tilde{Y}_{n,1}, \dots, \tilde{Y}_{n,n-1}, Z_n) \in A)}{\mathbb{Q}(B_k)}.
        \end{align*}
        But also,
        \begin{align*}
            \mathbb{Q}(B_k) &= \mathbb{Q}(\{ |\tilde{Y}_{(n,n-k-1)}| \le |Z_n| \le |\tilde{Y}_{(n, n-k)}|\})\\
            &=\binom{n-1}{k} \mathbb{E}\left[\left(F_{R}(|Z_n|)\right)^{n-k-1}\left(1 - F_{R}(|Z_n|)\right)^k\right] \\
            &= \binom{n-1}{k} \frac{n!}{(n-k-1)!k!}\int_{0}^{\infty} [F_{R}(z)]^{2n-2k-2}[1 - F_{R}(z)]^{2k} f_{R}(z)\ dz\\
            % &= \binom{n-1}{k} \frac{n!}{(n-k-1)!k!}\int_{0}^{1} u^{2n-2k-2}(1 - u)^{2k} \ dz\\  
            &= \binom{n-1}{k} \frac{n!}{(n-k-1)!k!}B(2n-2k-1, 2k+1) \\
            % &= \frac{(n-1)!}{(n-k-1)!k!}\frac{n!}{(n-k-1)!k!}\frac{(2n-2k-2)!(2k)!}{(2n-1)!}\\
            &\ge \frac{1}{2\cdot 4^k}\binom{2k}{k} \ge \frac{1}{4\sqrt{k+1}}.
        \end{align*}
    
        And thus the result follows.
    \end{proof}

    We will use Proposition \ref{propo:rejection} extensively in the following to draw conclusions for the $(n-k)-$th order statistic $Y_{(n-k)}$ using the results for $\{\tilde{Y}_{1, n-1}, \dots, \tilde{Y}_{n,n-1}, Z_n\}$.

    \subsection{Proof of Proposition~\ref{propo:evt}}

\label{sec:max_scale}

        In this section, we will prove Proposition \ref{propo:evt}. 
        
        A note on the notations before we begin. In this and following sections, we will use `$f \sim g$' for two functions $f, g: (0, \infty) \to (0, \infty)$ to indicate $\lim_{x \to \infty} f(x)/g(x) = 1$. 
        
        For any function $f$, we will use `$f \le RV_{\rho}$' to mean that there exists a $g \in RV_{\rho}$ and $x_0$ such that for all $x \ge x_0$, $f(x) \le g(x)$. Throughout this section we would often be working in the regime $|x| \to \infty$ where $x \in \mathbb{R}^d$. Accordingly we will write `$\le RV_{\rho}(|x|)$' to highlight this regime. 
    
        Let $h$ and $t_1$ be as in Assumption \ref{assum:density}. Since $h \in RV_{-\beta-d}$, there exists a slowly varying $L$\footnote{A measurable function $h:(0, \infty) \to (0, \infty)$ is slowly varying if, for all $r > 0$, $\lim_{t \to \infty} h(r t)/h(t) = 1$.} such that $h(t) = t^{-\beta-d}L(t)$ for all $t>0$. By assumption, $h$ i.e. $t \mapsto t^{-\beta-d}L(t)$ is non-increasing on $[t_1, \infty)$. Additionally, by Potter bounds for slowly varying functions (see \cite[][, Theorem 1.5.6]{bingham_regular_1987}), there exists a $t_2 > 0$ such that for all $t > t_2$, $L(t/2) \le 2L(t)$. Let $t_0 = \max\{t_1, t_2\}$.
        
        We begin with observing that the distribution of the radial component lies in the domain of attraction of an extreme value distribution. 

        \begin{proposition}
        \label{propo:doa}
            Let $P$ satisfy Assumption \ref{assum:density} and $F_{R}$ denote the distribution function of the radial component. Let $A = \omega_{d-1}K(\beta, d)/\beta$ be a constant. Then,
            \begin{equation}
            \label{eq:evt}
                F^n_{R}((An)^{1/\beta} r) \to \begin{cases}
                    \exp(-r^{-\beta}), & r > 0, \\
                    0, & r \le 0.
                \end{cases}
            \end{equation}
            as $n \to \infty$.
        \end{proposition}

        \begin{proof}
            From \eqref{eq:radial_df}, note that, 
            \[
                1 - F_{R}(r) = \omega_{d-1}\int_{r}^{\infty} t^{d-1}h(t)\ dt, \quad r > 0 
            \]
            and $1 - F_R(r) = 1$ for all $r \le 0$. Since $h \in RV_{-\beta-d}$, it follows that $t^{d-1}h(t) \in RV_{-\beta-1}$. For $\beta > 0$, it follows by the Karamata's Theorem (see \cite[][Theorem 0.6]{resnick_extreme_2008}) that, $1 - F_{R} \in RV_{-\beta}$. By Proposition 1.11 of \cite{resnick_extreme_2008}, there exist a sequence $a_n$ such that,
            \[
                F^n_{R}(a_nr) \to \begin{cases}
                \exp(-r^{-\beta}), & r > 0, \\
                0, & r \le 0.
            \end{cases}
            \]
            It remains to show that $a_n$ can be chosen to be $(An)^{1/\beta}$. 
            
            Also by Karamata's theorem (\cite[][Theorem 0.6]{resnick_extreme_2008}) we get the von-Mises condition,
            \begin{equation}
            \label{eq:von_mises}
                \lim_{r \to \infty} \frac{rF'_{R}(r)}{1 - F_{R}(r)} = \beta,
            \end{equation}
            where recall that $F'_{R}(r) = f_{R}(r) = \omega_{d-1}r^{d-1}h(r)$ is the density function of the radial component. By Proposition 1.15 of \cite{resnick_extreme_2008}, $a_n$ can be chosen to satisfy $a_nF'_R(a_n) \sim \beta/n$. Define, $a_n = (An)^{1/\beta}$. Then,
            \[
                a_nF'_{R}(a_n) = \omega_{d-1}(a_n)^dh(a_n) = \omega_{d-1} (a_n)^{-\beta} \left(a^{\beta+d}_nh(a_n)\right) = \frac{\omega_{d-1}}{An}\left(a^{\beta+d}_nh(a_n)\right).
            \]
            And so,
            \[
                \lim_{n \to \infty} \frac{a_nF'_{R}(a_n)}{\beta/n} = \lim_{n \to \infty} \frac{\omega_{d-1}}{A\beta}\left(a^{\beta+d}_nh(a_n)\right) = 1,
            \]
            since $a_n \to \infty$ as $n \to \infty$ and $t^{\beta+d}h(t)$ converges to $K(\beta, d)$ by assumption. 
            The result follows.
        \end{proof}

        We now give a simple upper bound on ball probability in remote locations of the space. This will be used in the proof of the second part of Proposition \ref{propo:evt}. 
        
        \begin{lemma}
        \label{lemma:ball_probability}
            Let $Y \sim P$ where $P$ satisfies Assumption \ref{assum:density}. Let $z$ be an arbitrary point in space such that $|z| > t_0$ and let $r < |z|/2$. Then,
            \[
                \PP(|Y - z| \le r) \le \frac{\omega_{d-1}}{d\cdot 2^{-\beta-d-1}}\cdot r^d|z|^{-\beta-d} L(|z|)
            \]
            where $L \in RV_{0}$ satisfies $h(t) = t^{-\beta-d}L(t)$. 
        \end{lemma}
    
        \begin{proof}
            Let $A_r(z)$ denote the ball $\{y: |y-z| < r\}$ of radius $r$ centered at $z$. Then,
            \[
                \PP(|Y-z| \le r) = \int_{A_r(z)} h(|y|)dy \le \text{Vol}(A_r(z))\cdot\sup_{A_r(z)}h(|y|) \le \frac{\omega_{d-1}r^d}{d}\cdot\sup_{A_r(z)}h(|y|).
            \]
            Observe that $\inf\{|y|: y \in A_{r}(z)\} > |z|/2$ which means that $2|y| > |z| > t_0$ for all $y \in A_{r}(z)$. Moreover, since $L(t/2) \le 2L(t)$ for all $t > t_0 \ge t_2$, it follows that,
            \[
                h(|y|) = |y|^{-\beta-d}L(|y|) \le 2|y|^{-\beta-d}L(2|y|) = 2^{\beta+d+1} h(2|y|).
            \]
            But also by assumption $h(t) = t^{-\beta-d}L(t)$ is non-increasing on $[t_0, \infty)$. Since $2|y| > |z| > t_0$, we get, $h(|y|) \le 2^{\beta + d +1}h(|z|)$ for all $y \in A_{r}(z)$. Thus,
            \[
                \PP(|Y-z| \le r) \le \frac{\omega_{d-1}r^d}{d}\cdot 2^{\beta + d +1}h(|z|).
            \]
            Hence, proved.
        \end{proof}

        \begin{proof}[Proof of Proposition \ref{propo:evt}]
            By Proposition \ref{propo:doa}, the distribution function of the radial component of $P$ lies in the domain of attraction of a Fréchet distribution with parameter $\beta > 0$ and with normalizing sequence $a_n = (An)^{1/\beta}$. By Theorem 2.1.1 of \cite{de_haan_extreme_2006}, it follows that, as $n \to \infty$,
            \[
                \frac{|Y_{(n-k)}| - (An)^{1/\beta}}{\beta^{-1}(An)^{1/\beta}} \xrightarrow{d}  \frac{(E_1 + \dots + E_{k+1})^{-1/\beta} - 1}{\beta^{-1}},
            \]
            where $E_1, \dots, E_{k+1}$ are i.i.d. standard exponential. Simplifying the terms gives,
            \[
                \frac{|Y_{(n-k)}|}{n^{1/\beta}} \xrightarrow{d}  \frac{(E_1 + \dots + E_{k+1})^{-1/\beta}}{A^{-1/\beta}},
            \]
            This proves the first part of the statement. The almost sure convergence follows by considering a rich enough probability space.

            For the second part, let $0 < \epsilon < 1$ be arbitrary. Define $g: (\mathbb{R}^d)^n \to \mathbb{R}$ by,
            \[
                g(y_1, \dots, y_n) = \frac{1}{|y_n|}\min_{j\neq n}|y_j - y_n|.
            \]
            Note that $g$ is symmetric in its first $n-1$ arguments. And,
            \[
                \PP\left(\min_{j \ne {n-k}}|Y_{(j)} - Y_{(n-k)}| > |Y_{(n-k)}|/2n^{\epsilon}\right) = \PP\left(g(Y_{-(n-k)}, Y_{(n-k)}) > 1/2n^{\epsilon}\right).
            \]
            Let $\tilde{Y}_{n,1}, \dots, \tilde{Y}_{n,n-1}, Z_n$ and $\QQ$ be as defined in the Section~\ref{sec:order_stats}. By Proposition \ref{propo:rejection}, 
            \begin{align*}
                &\PP\left(\min_{j \ne {n-k}}|Y_{(j)} - Y_{(n-k)}| > |Y_{(n-k)}|/2n^{\epsilon}\right) \\
                &\quad \quad = 1 - \PP\left(g(Y_{-(n-k)}, Y_{(n-k)}) \le 1/2n^{\epsilon}\right) \\
                &\quad \quad \ge 1 - 4\sqrt{k+1} \cdot \mathbb{Q}\left(g(\tilde{Y}_{1, n-1}, \dots, \tilde{Y}_{n,n-1}, Z_n) \le 1/2n^{\epsilon}\right)\\
                &\quad \quad = 1 - 4\sqrt{k+1} \cdot \mathbb{Q}\left(\min_{j}|\tilde{Y}_{n, j} - Z_n| \le |Z_n|/2n^{\epsilon}\right).
            \end{align*}
            We will show that $\mathbb{Q}\left(\min_{j}|\tilde{Y}_{n, j} - Z_n| \le |Z_n|/2n^{\epsilon}\right)$ goes to $0$ as $n \to \infty$. Consider,
            \begin{align*}
                &\mathbb{Q}\left(\min_{j}|\tilde{Y}_{n, j} - Z_n| \le |Z_n|/2n^{\epsilon}\right)\\&\quad \quad \quad = \mathbb{Q}\left(\exists j :|\tilde{Y}_{n, j} - Z_n| \le |Z_n|/2n^{\epsilon}\right) \\
                &\quad \quad \quad = \mathbb{E}\left[\mathbb{Q}\left(\min_{j}|\tilde{Y}_{n, j} - Z_n| \le |Z_n|/2n^{\epsilon}\bigg|Z_n\right)\right]\\
                &\quad \quad \quad \le \mathbb{E}\left[n\mathbb{Q}\left(|\tilde{Y}_{n, 1} - Z_n| \le |Z_n|/2n^{\epsilon}\bigg|Z_n\right)\right]\\
                &\quad \quad \quad \le n\mathbb{Q}(|Z_n| \le 2t_0) + n\mathbb{E}\left[n\mathbb{Q}\left(|\tilde{Y}_{n, 1} - Z_n| \le |Z_n|/2n^{\epsilon}\bigg|Z_n, |Z_n| > 2t_0\right)\right].
            \end{align*}
            But from Lemma \ref{lemma:ball_probability},
            \begin{align*}
                \mathbb{Q}\left(|\tilde{Y}_{n, 1} - Z_n| \le |Z_n|/2n^{\epsilon}\bigg|Z_n, |Z_n| > 2t_0\right) &\le \frac{\omega_{d-1}}{d\cdot2^{-\beta-d-1}}\left(\frac{|Z_n|}{2n^{\epsilon}}\right)^d|Z_n|^{-\beta-d} L(|Z_n|) \\
                &= \frac{2^{\beta+1}\omega_{d-1}}{d}\frac{|Z_n|^{-\beta}}{n^{d\epsilon}}L(|Z_n|).
            \end{align*}
            Combining with above,
            \[
                \mathbb{Q}\left(\min_{j}|\tilde{Y}_{n, j} - Z_n| \le |Z_n|/2n^{\epsilon}\right) \le n\mathbb{Q}(|Z_n| \le 2t_0) +  \frac{2^{\beta+1}\omega_{d-1}}{d}\mathbb{E}\left[\frac{|Z_n|^{-\beta}}{n^{d\epsilon-1}}L(|Z_n|)\right].
            \]

            For the first term, observe that
            \[
                n\mathbb{Q}(|Z_n| \le 2t_0) = n \sum_{j=0}^k \binom{n}{j} (F_{R}(2t_0))^{n-j}(1 - F_R(2t_0))^{n-j} \le (k+1)n^{k+1}(F_R(2t_0))^{n-k},
            \]
            which goes to $0$ as $n \to \infty$ since $F_{R}(2t_0) < 1$. 
            For the second term, recall that $|Z_n|$ scales as $n^{1/\beta}$, $L(t)$ converges to a finite constant, and $d\epsilon > 0$. Then, the integrand converges to $0$ as $n \to \infty$. Consequently, the expectation goes to $0$ by the bounded convergence theorem. This proves the result.
        \end{proof}

    \section{Proof of Theorem~\ref{theo:no_roots}}
    \begin{proof}[Proof of Theorem~\ref{theo:no_roots}]
    We show it for $k =0$ i.e. $Y_{(n)}$. The arguments for $k > 0$ will follow similarly. Let $x \in B_{\sqrt{\nu}}(Y_{(n)})$ be arbitrary and let $v = -x/|x|$. Consider,
    \begin{align*}
        v^TS_n(x) = \sum_{j=1}^n \frac{-x^T(x - Y_j)/|x|}{\nu + |x - Y_j|^2} &\le \sum_{j=1}^{n} \frac{-|x| + |Y_{(j)}|}{\nu + |x - Y_{(j)}|^2}.
    \end{align*}
    Consider the set $H_n = \{j: |Y_{(j)}| > |Y_{(n)}|/2 - \sqrt{\nu}\}$. For $j \notin H_n$ and $x \in B_{\sqrt{\nu}}(Y_{(n)})$, $-|x| + |Y_{j}| \le -|Y_{(n)}|/2$ since $|x| \ge |Y_{(n)}| - \sqrt{\nu}$. Similarly, $|x - Y_{j}| \le 3|Y_{(n)}|/2$ since $|x| \le |Y_{(n)}| + \sqrt{\nu}$. And so, 
    \begin{align*}
        v^TS_n(x) &\le \sum_{j\notin H_n} \frac{-|x| + |Y_{(j)}|}{\nu + |x - Y_{(j)}|^2} + \sum_{j \in H_n}\frac{-|x| + |Y_{(j)}|}{\nu + |x - Y_{(j)}|^2}\\ 
        &\le \sum_{j\notin H_n} \frac{-2|Y_{(n)}|}{4\nu + 9|Y_{(n)}|^2} + \sum_{j \in H_n}\frac{1}{2\sqrt{\nu}}
    \end{align*}
    for $x \in B_{\sqrt{\nu}}(Y_{(n)})$ and $v = -x/|x|$. Moreover since $|Y_{(n)}| \uparrow \infty$, for large $n$, $4\nu \le 9|Y_{(n)}|^2$ and we get,
    \[
        v^TS_n(x) \le -\frac{n - \sharp H_n}{9|Y_{(n)}|} + \frac{\sharp H_n}{2\sqrt{\nu}}.
    \]
    Let $K \ge 1$ be arbitrary. Then, on the event $\{\sharp H_n \le K\}$, $v^TS_n(x) \downarrow -\infty$ as $n \to \infty$ by Proposition~\ref{propo:evt} for $\beta > 1$. In particular, for any $c > 0$ and for any $K \ge 1$, there exists an $N_0(K)$ such that for $n \ge N_0$
    \begin{align*}
        \{\sharp H_n \le K\} &\implies \left\{\inf \left\{(-x/|x|)^TS_n(x): x \in B_{\sqrt{\nu}}(Y_{(n)})\right\} \le -c\right\} \\
        &\implies \left\{\forall x \in B_{\sqrt{\nu}}(Y_{(n)}) \exists v \in \mathbb{S}^{d-1} : v^TS_n(x)\le -c\right\}. \numberthis \label{eq:neg1}
    \end{align*}

    For $K \ge 1$, consider,
    \begin{align*}
         \PP(\sharp H_n > K) &= \PP\left(\sharp\{j \le n-1: |{Y}_{(j)}| > |Y_{(n))}|/2 - \sqrt{\nu}\} \ge K\right) \\
         &\le \frac{1}{K}\mathbb{E}\left(\sharp\{j \le n-1: |{Y}_{(j)}| > |Y_{(n))}|/2 - \sqrt{\nu}\}\right) \\
         &\le \frac{n}{K}\mathbb{E}\left(\frac{1 - F_{R}(|Y_{(n)}|/2 - \sqrt{\nu})}{F_{R}(|Y_{(n)}|/2)}\right),
    \end{align*}
    where $F_{R}$ is the distribution function of the radial component of $Y$ under $Y \sim P$.
    The term inside the expectation is monotonically decreasing in $|Y_{(n)}|$. By Proposition~\ref{propo:evt}, for some constant $C_{\beta, d}$, 
    \[
        n\left(\frac{1 - F_{R}(|Y_{(n)}|/2 - \sqrt{\nu})}{F_{R}(|Y_{(n)}|/2)}\right) \longrightarrow C_{\beta, d} \left(\frac{G_k}{2}\right)^{-\beta}
    \]
    as $n \to \infty$, $\PP-$almost surely. This means that for large $n$,
    \[
        \PP(\sharp H_n > K) \le \frac{1}{K}C'_{\beta, d} \left(\frac{G_k}{2}\right)^{-\beta}
    \]
    for some $C'_{\beta, d} > C_{\beta, d}$. The right-hand side is decreasing in $K$. Let $\epsilon > 0$ be arbitrary. Then, we may chose $K(\epsilon) \ge 1$ such that the right-hand side is less than $\epsilon$ i.e., given an arbitrary $\epsilon > 0$, we may chose $N_1 > N_0$ such that,
    \[
        \PP(\sharp H_n > K(\epsilon)) \le \epsilon, 
    \]
    for all $n \ge N_1$. Combining this with \eqref{eq:neg1}, we get for all $n \ge N_1$,
    \[
        \PP\left(\left\{\forall x \in B_{\sqrt{\nu}}(Y_{(n)}) \exists v \in \mathbb{S}^{d-1} : v^TS_n(x)\le -c\right\}\right) \ge 1-\epsilon.
    \]
    Since $\epsilon > 0$ is arbitrary, the result follows. 
\end{proof}
    \section{Empirical score approximation}
\label{sec:score_approximation_proof}
    In this Section, we will prove the results from Section~\ref{sec:score_approximation}. 

    \subsection{Proof of Lemma~\ref{lemma:score_bound}}

    To prove Lemma~\ref{lemma:score_bound} split $\mathbb{S}_n$ into,
    \begin{align*}
        \mathbb{S}_n(z) &\le \sup_{x:|x-z| \le \frac{2|z|}{n}} \frac{2|z|}{n}\left|S_n(x) - n\mathbb{E}S(x, Y)\right| + \sup_{x:|x-z| \le \frac{2|z|}{n}}\frac{2|z|}{n}\left|n\mathbb{E}S(x, Y) - \frac{nx}{|x|^2}\right|
    \end{align*}
    We will control each term individually in the following two lemmas. 

     \begin{lemma}
     \label{lemma:concentration}
        Let $0 < \epsilon <  1$ be arbitrary. There exists a $T_2(\epsilon)$ such that for all $n \ge T_2(\epsilon)$,  
        \begin{align*}
            &\Prob\left(\sup_{x:|x-z| \le 2|z|/n}\frac{2|z|}{n}\bigg|S_n(x) - n\mathbb{E}S(x, Y)\bigg| > 48\nu n^{2\epsilon-1} + \frac{|z|^{-\beta}}{n^{\epsilon(d-1)}}L'(|z|) + t\right)\\
            &\quad \quad \le 5^d\exp\left(-\frac{n^{1-2\epsilon}t^2}{512}\right) + C_d\frac{5^d|z|^{-\beta}}{2^dn^{d\epsilon-1}}L(|z|),
        \end{align*}  
        for all $t > 0$ and $|z| > 2n\sqrt{\nu}$, where $L'$ is slowly varying and $L(t) = t^{\beta+d}h(t)$.
    \end{lemma}

    \begin{proof}
        Let $0 < \epsilon < 1$ be arbitrary. Define $\delta_n = 2/n + 1/2n^{\epsilon}$. Let $T_1(\epsilon) > t_0/\sqrt{\nu}$ be such that for all $n \ge T_1(\epsilon)$, $\delta_n < 1/2$. Let $n > T_1(\epsilon)$ be arbitrary. In the following $|z| > 2n\sqrt{\nu}$. Then we have:

        \begin{enumerate}
            \item $|z|/2 > n\sqrt{\nu} > T_1(\epsilon)\sqrt{\nu} > t_0$.
            \item $|z|/2n^{\epsilon} > 1$.
            \item $|z| - |z|\delta_n > |z|/2$.
        \end{enumerate}

        Define the set $A_{\epsilon} = \{y: |y - z| < |z|\delta_n\}$ and $A_{\gamma} = \{y: |y - z| < 2|z|/n\}$. 
        Let $S_1(x, y) = S(x, y)1(y \notin A_{\epsilon})$ and $S_2(x, y) = S(x, y)1(y \in A_{\epsilon})$. 
        Define,
        \begin{equation}
        \label{eq:psi_n}
            \Psi_n(x) = n^{-1}|z|\left(S_n(x) - n\mathbb{E}S(x, Y)\right), \quad x \in \mathbb{R}^d,
        \end{equation}
        Then,
        \begin{equation}
        \label{eq:psi_split}
            \Psi_n(x) = \underbrace{\frac{2|z|}{n}\sum_{j=1}^n (S_1(x, Y_j) - \mathbb{E}S_1(x, Y_j))}_{:= \Psi_{1, n}(x)} + \overbrace{\frac{2|z|}{n}\sum_{j=1}^n (S_2(x, Y_j) - \mathbb{E}S_2(x, Y_j))}^{:= \Psi_{2, n}(x)}
        \end{equation}
        We will consider each term separately.

    \noindent
    {\it \underline{1) Bounds on $\Psi_{1, n}:$}}

        {\it \underline{1.1) Discretization of $A_{\gamma}$:}}
        
        First, consider $\Psi_{1,n}$. Consider the function $s: x \mapsto x/(\nu + |x|^2)$. Suppose $|x|, |y| > r > 1$. Then,
        \begin{align*}
            |s(x) - s(y)| &= \left|\frac{x}{\nu + |x|^2} - \frac{y}{\nu + |y|^2}\right| \le \frac{1}{(\nu + |x|^2)(\nu + |y|^2)}\left|\nu(x - y) + x|y|^2 - y|x|^2\right|\\
            &\le \frac{\nu|x - y|}{(\nu + |x|^2)(\nu + |y|^2)} + \frac{|y|^2|x - y|}{(\nu + |x|^2)(\nu + |y|^2)} + \frac{|y|||x|^2 - |y|^2|}{(\nu + |x|^2)(\nu + |y|^2)} \\
            &\le \frac{\nu|x-y|}{r^4} + \frac{|x-y|}{r^2} + \frac{|x-y|}{r^3} \le 3\nu\frac{|x-y|}{r^2}. \numberthis \label{eq:s_lipschitz}
        \end{align*}
        
        Suppose $y \notin A_{\epsilon}$ and $x \in A_{\gamma}$, then, $|x - y| > ||y - z| - |x - z|| \ge |z|\delta_n - 2|z|/n = |z|/2n^{\epsilon} > 1$. Equation \eqref{eq:s_lipschitz} implies for any $x^1, x^2 \in A_{\gamma}$,
        \begin{align*}
            |S_1(x^1, Y) - S_1(x^2, Y)| \le \frac{12\nu n^{2\epsilon}}{|z|^{2}} ||x^1 - Y| - |x^2 - Y||\le \frac{12\nu n^{2\epsilon}|x^1 - x^2|}{|z|^{2}}
        \end{align*}
        with probability $1$. Now, consider,
        \begin{align*}
            &\left|\Psi_{1, n}(x^1) - \Psi_{1, n}(x^2) \right| \\
            &\quad \quad =\frac{2|z|}{n}\left|\sum_{j=1}^n (S_1(x^1, Y_j) - \mathbb{E}S_1(x^1, Y_j)) - \sum_{j=1}^n (S_1(x^2, Y_j) - \mathbb{E}S_1(x^2, Y_j))\right| \\
            &\quad \quad \le \frac{2|z|}{n}\left(\sum_{j=1}^n\left|S_1(x^1, Y_j) - S_1(x^2, Y_j)\right| + n\mathbb{E}\left[\left|S_1(x^1, Y) - S_1(x^2, Y)\right|\right]\right)\\
            &\quad \quad \le \frac{2|z|}{n}\cdot 2n \cdot\left(\frac{12\nu n^{2\epsilon}|x^1 - x^2|}{|z|^{2}}\right) \le 48\nu n^{2\epsilon}|z|^{- 1}|x^1 - x^2|,
        \end{align*}
        with probability $1$.
        It then follows that for all $x \in A_{\gamma}$,
        \[
            |\Psi_{1, n}(x)| \le \left|\Psi_{1, n}(x) - \Psi_{1, n}(z)\right| + |\Psi_{1, n}(z)|  \le 48\nu n^{2\epsilon}|z|^{- 1}\frac{|z|}{n} + |\Psi_{1, n}(z)|.
        \]
        Thus,
        \begin{equation}
        \label{eq:psi1_break1}
            \sup_{x \in A_{\gamma}} |\Psi_{1, n}(x)| \le 48\nu n^{2\epsilon-1} + |\Psi_{1, n}(z)|
        \end{equation}
        with probability $1$.

        {\it \underline{1.2) Concentration in $\Psi_{1,n}$ }}

        Observe that since $|y - z| \ge |z|\delta_n > |z|/2n^{\epsilon}$ for all $y \notin A_{\epsilon}$, it follows that $|S_1(z, Y)| \le 2n^{\epsilon}|z|^{-1}$ is a bounded random vector. Standard concentration inequality via Euclidean nets (see e.g. Proposition 28 in \cite{bordino_tests_2024}) yields,
        \begin{align*}
            \Prob(|\Psi_{1, n}(z)| > t) &= \Prob\left(\frac{1}{n}\left|\sum_{j=1}^n (S_1(z, Y_j) - \mathbb{E}S_1(z, Y_j))\right| > \frac{t}{2|z|}\right)\\
            &\le 5^d\exp\left(-\frac{n}{8}\cdot\frac{t^2}{4|z|^2}\frac{1}{(2n^{\epsilon}|z|^{-1})^2}\right) = 5^d\exp\left(-\frac{n^{1-2\epsilon}t^2}{128}\right). \numberthis \label{eq:psi1_break2}
        \end{align*}
         
        {\it \underline{1.3) Combining observations}}

        Using \eqref{eq:psi1_break2} and \eqref{eq:psi1_break1}, we get for all $z$ such that $|z| \ge 2n\sqrt{\nu}$,
        \begin{equation}
        \label{eq:break1}
            \Prob\left(\sup_{x \in A_{\gamma}} |\Psi_{1, n}(x)| > 48\nu n^{2\epsilon-1} + t/2\right) \le 5^d\exp\left(-\frac{n^{1-2\epsilon}t^2}{128}\right).
        \end{equation}

    \noindent
    {\it \underline{2) Bounds on $\Psi_{2, n}:$}}
    
        Now, we consider $\Psi_{2,n}$. From \eqref{eq:psi_split} we get,
        \begin{equation}
        \label{eq:psi2_split}
            \left|\Psi_{2,n}(x)\right| \le  \frac{2|z|}{n} \left|\sum_{j=1}^n S_2(x, Y_j)\right| + 2|z|\cdot\mathbb{E}|S_2(x, Y)|.
        \end{equation}

        {\it \underline{2.1) Controlling the expectation in \eqref{eq:psi2_split}}}
        
        For any $x \in A_{\gamma}$, consider first the expectation $\mathbb{E}|S_2(x, Y)|$. Then,
        \begin{align*}
            \mathbb{E}|S_2(x, Y)| &= \int_{A_{\epsilon}}|S(x, y)|P(dy) \\
            &= \int_{A^1_{\epsilon}}|S(x, y)| P(dy) + \int_{A^2_{\epsilon}}|S(x, y)| P(dy),
        \end{align*}
        where $A^1_{\epsilon}  = \{y: |x - y| \le \sqrt{\nu}\} \cap A_{\epsilon}$ and $A^2_{\epsilon} = A_{\epsilon}\backslash A^1_{\epsilon}$. 
        
        On $A^1_{\epsilon}$, we use the global bound on $|S(x, y)|$ i.e. $|S(x, y)| \le 1/2\sqrt{\nu} =: M$. And so,
        \begin{align*}
            \int_{A^1_{\epsilon}}|S(x, y)|P(dy) &\le MP(\{Y: |x - Y| \le \sqrt{\nu}\} \cap A_{\epsilon})\le MP(\{Y:|x - Y|\le \sqrt{\nu}\})\\ 
            &\le MC_d(\sqrt{\nu})^d |x|^{-\beta-d}L(|x|),
        \end{align*} 
        where the last inequality is due to Lemma \ref{lemma:ball_probability} and because for all $x \in A_{\gamma}$, $|x| > |z| - 2|z|/n > |z|/2 > t_0$ and $\sqrt{\nu} < n\sqrt{\nu} < |x|/2$. Recall that $|x|^{-\beta-d}L(|x|) = h(|x|)$. Using the assumption that $h$ is non-increasing after $t_0$, we get, 
        \begin{equation}
        \label{eq:expectation_s2_first}
            \sup_{x \in A_{\gamma}}\int_{A^1_{\epsilon}}|S(x, y)|P(dy) \le \frac{MC_d\nu^{d/2}}{2^{-\beta-d}}|z|^{-\beta - d}L(|z|/2) \le \frac{MC_d\nu^{d/2}}{2^{-\beta-d-1}}|z|^{-\beta - d}L(|z|).
        \end{equation}
        
        On the other hand, consider,
        \[
            \int_{A^2_{\epsilon}} \frac{|x-y|}{\nu + |x-y|^2} P(dy) \le \int_{A^2_{\epsilon}}\frac{1}{|x-y|}h(|y|)dy < \infty,
        \]
        since $|x-y| > \sqrt{\nu}$ on $A^2_{\epsilon}$. For all $y \in A^2_{\epsilon}$, $|y| > |z| - |z|\delta_n > |z|/2$. And so, for all $y \in A^2_{\epsilon}$, $h(|y|) \le h(|z|/2)$. Moreover, on $A^2_{\epsilon}$, note that for any $x\in A_{\gamma}$, $\sqrt{\nu} < |x - y| < 2|z|/n + |z|\delta_n$. Recall that $\delta_n = 2/n + 1/2n^{\epsilon}$. Since $0<\epsilon<1$, we have $2|z|/n + |z|\delta_n \le 9|z|/2n^{\epsilon}$. Then,
        \begin{align*}
            \int_{A^2_{\epsilon}}\frac{1}{|x-y|}h(|y|)dy &\le h(|z|/2)\int_{A^2_{\epsilon}}\frac{1}{|x-y|}dy \le 2^{\beta+d+1}h(|z|) \int_{\sqrt{\nu} \le |x-y| \le \frac{9|z|}{2n^{\epsilon}}}\frac{1}{|x-y|}dy\\
            &=2^{\beta+d+1}h(|z|)\begin{cases}
                \omega_d\displaystyle\int_{\sqrt{\nu}}^{\frac{9|z|}{2n^{\epsilon}}}r^{d-2}\ dr, & d \ge 2,\\
                2\displaystyle\int_{\sqrt{\nu}}^{\frac{9|z|}{2n^{\epsilon}}}\frac{1}{r}\ dr, & d =1,
            \end{cases}\\
            &\le 2^{\beta+d+1}h(|z|)\begin{cases}
                \omega_d\left(\frac{9|z|}{2n^{\epsilon}}\right)^{(d-1)}, & d \ge 2,\\
                2\ln\frac{9|z|}{2n^{\epsilon}}, & d =1,
            \end{cases} \numberthis \label{eq:expectation_s2_second}
        \end{align*}
        where $\omega_d$ is the surface area of a $(d-1)-$dimensional unit sphere in $\mathbb{R}^d$. From \eqref{eq:expectation_s2_first} and \eqref{eq:expectation_s2_second}, we get,
        \[
            2|z|\cdot\sup_{x \in A_{\gamma}}\mathbb{E}|S_2(x, Y)| \le 8\left(\frac{|z|}{2}\right)^{-\beta-d+1}L(|z|)\begin{cases}
                MC_d\nu^{d/2} + \omega_d\left(\frac{9|z|}{2n^{\epsilon}}\right)^{(d-1)}, & d \ge 2,\\
                MC_d\nu^{d/2} + 2\ln\frac{9|z|}{2n^{\epsilon}}, & d =1,
            \end{cases}.
        \]
        Note that since $|z| > 2n\sqrt{\nu}$, the term $|z|/n^{\epsilon}$ can be made arbitrarily large by taking $n$ large. And so, let $T_2(\epsilon) > T_1(\epsilon)$ be such that for all $n > T_2(\epsilon)$, there exists a slowly varying function $L'$ such that,
        \begin{equation}
        \label{eq:psi2_break1}
            |z|\cdot\sup_{x \in A_{\gamma}}\mathbb{E}|S_2(x, Y)| \le \frac{|z|^{-\beta}}{n^{\epsilon(d-1)}}L'(|z|).
        \end{equation}

        {\it \underline{2.2) Controlling the average in \eqref{eq:psi2_split}}}

        Observe that $|\sum_{j=1}^n S_2(x, Y_j)| > 0$ if and only if there exists a $Y_j$ such that $Y_j \in A_{\epsilon}$. Recall that $A_{\epsilon} = \{y: |y-z| \le |z|\delta_n\}$. Since $\delta_n < 1/2$ we get by Lemma \ref{lemma:ball_probability},
        \[
            \Prob\left(\sup_{x \in A_{\gamma}}\left|\sum_{j=1}^n S_2(x, Y_j)\right| > \frac{nt}{4|z|}\right) \le nP(|Y-z|\le|z|\delta_n) \le nC_d(|z|\delta_n)^d|z|^{-\beta-d}L(|z|).
        \]
        But also $\delta_n = 2/n + 1/2n^{\epsilon} \le 5/2n^{\epsilon}$. So,
        \begin{equation}
        \label{eq:psi2_break2}
            \Prob\left(\sup_{x \in A_{\gamma}}\left|\sum_{j=1}^n S_2(x, Y_j)\right| > \frac{nt}{4|z|}\right) \le C_d\frac{5^d|z|^{-\beta}}{2^dn^{d\epsilon-1}}L(|z|).
        \end{equation}

        {\it \underline{2.3) Combining observations}}
        
        Combining \eqref{eq:psi2_break1} and \eqref{eq:psi2_break2} with \eqref{eq:psi2_split}, we get,
        \begin{equation}
        \label{eq:break2}
            \Prob\left(\sup_{x \in A_{\gamma}}|\Psi_{2, n}(x)| > \frac{|z|^{-\beta}}{n^{\epsilon(d-1)}}L'(|z|) + t/2\right) \le C_d\frac{5^d|z|^{-\beta}}{2^dn^{d\epsilon-1}}L(|z|).
        \end{equation}

    \noindent
    {\it \underline{3) Combining $\Psi_{1,n}$ and $\Psi_{2, n}$}}
            
        Suppose $n \ge T_2(\epsilon)$. Combining \eqref{eq:break1} and \eqref{eq:break2}, we get, for all $|z| > 2n\sqrt{\nu}$,
        \begin{align*}
            &\Prob\left(\sup_{x:|x-z| \le 2|z|/n}\frac{2|z|}{n}\bigg|S_n(x) - n\mathbb{E}S(x, Y)\bigg| > 48\nu n^{2\epsilon-1} + \frac{|z|^{-\beta}}{n^{\epsilon(d-1)}}L'(|z|) + t\right)\\
            &\quad \quad \le 5^d\exp\left(-\frac{n^{1-2\epsilon}t^2}{512}\right) + C_d\frac{5^d|z|^{-\beta}}{2^dn^{d\epsilon-1}}L(|z|).
        \end{align*}
        for a slowly varying $L'$. This finishes the proof.
    \end{proof}

    \begin{lemma}
    \label{lemma:error_expectation}
        Let $Y \sim P$ where $P$ satisfies Assumption \ref{assum:density} with $0 < \beta \le 1$. For all $\nu$, there exists a $T_3(\nu)$ such that for all $n \ge T_3(\nu)$ and $|z| > 2n\sqrt{\nu}$,
        \[
            \sup_{x:|x-z| \le \frac{2|z|}{n}}\frac{2|z|}{n}\left|n\mathbb{E}S(x, Y) - \frac{nx}{|x|^2}\right| \le 
                RV_{-\beta}(|z|).
        \]
    \end{lemma}

    \begin{proof} 
       Let $T_3(\nu) \ge 4$ be such that $T_3 \ge 2t_0/\sqrt{\nu}$. Let $n > T_3$ be arbitrary. Then, $|z| > 2n\sqrt{\nu}$ implies for all $|x - z| < 2|z|/n$, $|x| > |z|/2 > \max\{2t_0, 2\sqrt{\nu}\}$. For any such $x$, define
       \[
            g(y) = \frac{|x|(x - y)}{\nu + |x - y|^2}, \quad y \in \mathbb{R}^d.
       \]
       Then, $|x|\mathbb{E}S(x, Y) = \mathbb{E}g(Y)$, and,
       \begin{equation}
       \label{eq:first_break}
            \left||x|\mathbb{E}S(x, Y) - \frac{x}{|x|}\right| \le  \left|\mathbb{E}g(Y)1_{|Y| < |x|/2} - \frac{x}{|x|}\right| + \mathbb{E}|g(Y)|1_{|Y|\ge|x|/2}.
       \end{equation}
       We first consider $\mathbb{E}|g(Y)|1_{|Y|\ge|x|/2}$. We split the region $|y| \ge |x|/2$ into further three parts: 
        \[
            \{y: |y| \ge |x|/2\} = A_1 \cup A_2 \cup A_3,
        \]
        where,
        \begin{gather*}
            A_1 := \{|y| \ge |x|/2, |x - y| > |x|/2\}, \quad  
            A_2 := \{\sqrt{\nu} < |x - y| \le |x|/2\},\\
            A_3 := \{|x - y| \le \sqrt{\nu}\},
        \end{gather*}
        and consider each part separately.
        
        For $ y \in A_1$, observe that $|g(y)| \le 2$. Then,
        \[
            \mathbb{E}|g(Y)|1_{Y \in A_1} \le 2P(A_1) \le 1 - F_{R}\left(|x|/2\right),
        \]
        where $F_{R}$ is the distribution function of the radial component $|Y|$ and $1 - F_{R} \in RV_{-\beta}$. Thus, we have,
        \begin{equation}
        \label{eq:A1}
            \mathbb{E}|g(Y)|1_{Y \in A_1} \le RV_{-\beta}(|x|)
        \end{equation}
        
        For $y \in A_2$, note that $t_0 < |x|/2 \le|y| \le 3|x|/2$. Since $h$ is non-increasing on $[t_0, \infty)$,
        \begin{align*}
            \mathbb{E}|g(Y)|1_{Y \in A_2} = \int_{A_2} \frac{|x||x - y|}{\nu + |x - y|^2}h(|y|)\ dy \le |x|h(|x|/2)\int_{c <|u|< |x|/2} \frac{ |u|}{\nu + |u|^2}\ du.
        \end{align*}
        Now,
        \[
            \int_{\sqrt{\nu} <|u|< |x|/2} \frac{ |u|}{\nu + |u|^2}\ du \le \int_{\sqrt{\nu} <|u|< |x|/2} \frac{1}{|u|}\ du \le \begin{cases}
                C_d |x|^{d-1} & d \ge 2 \\
                C_1\ln |x|, & d = 1,
            \end{cases}
        \]
        for some dimension-dependent constants $C_d$. This implies,
        \[
            \mathbb{E}|g(Y)|1_{Y \in A_2} \le \begin{cases}
                C_d |x|^{d}h(|x|/2) & d \ge 2 \\
                C_1 |x|h(|x|/2)\ln |x|, & d = 1.
            \end{cases}
        \]
        Since $h \in RV_{-\beta-d}$ for $\beta > 0$, and $\ln \in RV_{0}$ i.e. $\ln(|x|)$ is slowly-varying, the right-hand side $\in RV_{-\beta}$ for all $d$. We get,
        \begin{equation}
        \label{eq:A2}
            \mathbb{E}|g(Y)|1_{Y \in A_2} \le RV_{-\beta}(|x|)
        \end{equation}

        Finally, consider $A_3$. Then, $|g(y)|$ is bounded by $|x|/2\sqrt{\nu}$. We get,
        \[
            \mathbb{E}|g(Y)|1_{Y \in A_3} \le \frac{|x|}{2\sqrt{\nu}} \PP(A_3).
        \]
        Note that $|x| \ge \max\{2t_0, 2\sqrt{\nu}\}$ implies $x - \sqrt{\nu} > t_0$. So $h$ is non-increasing on $|x|-\sqrt{\nu} \le |y| \le |x| + \sqrt{\nu}$. We get,
        \[
            \PP(A_3) = \int_{|x - y| \le \sqrt{\nu}} h(|y|)dy \le C'_d h(|x| - \sqrt{\nu}).
        \]
        And so, 
        \begin{equation}
        \label{eq:A3}
            \mathbb{E}|g(Y)|1_{Y \in A_3} \le RV_{-\beta-d}(|x|).
        \end{equation}

       Then, equations \eqref{eq:A1}, \eqref{eq:A2}, and \eqref{eq:A3} imply that,
       \begin{equation}
        \label{eq:second_break}
             \mathbb{E}|g(Y)|1_{|Y|\ge|x|/2} \le  \mathbb{E}|g(Y)|1_{Y \in A_1} + \mathbb{E}|g(Y)|1_{Y \in A_2} + \mathbb{E}|g(Y)|1_{Y \in A_3} \le RV_{-\beta}(|x|).
       \end{equation}
       
       Now, consider $|y| < |x|/2$. For $y \neq x$, we can decompose $g$ as,
        \[
            g(y) = \frac{x}{|x|} + \underset{:= g_1(y)}{\left(\frac{|x|(x - y)}{|x - y|^2} - \frac{x}{|x|}\right)} - \underset{:= g_2(y)}{\left(\frac{\nu|x|(x - y)}{|x - y|^2(\nu + |x - y|^2)}\right)}
        \]
        For $|y| < |x|/2$, it follows that $|x-y|> |x|/2$. It is then easy to see that $|g_2(y)| \le 8\nu/|x|^2$. Moreover,
        $|g_1(y)| \le 14\frac{|y|}{|x|}$. And so,
        \begin{equation*}
            \left|\mathbb{E}g(Y)1_{|Y| < |x|/2} - \frac{x}{|x|}\right| \le \frac{14}{|x|}\mathbb{E}\left(|Y|1_{|Y| < |x|/2}\right) + \frac{8\nu}{|x|^2}.
        \end{equation*}

        For $0 < \beta \le 1$, Karamata's theorem \citep{bingham_regular_1987} implies that $|x| \mapsto |x|^{-1}\mathbb{E}(|Y|1_{|Y| < |x|/2}) \in RV_{-\beta}$. So,
        \begin{equation}
        \label{eq:approx2}
            \left|\mathbb{E}g(Y)1_{|Y| < |x|/2} - \frac{x}{|x|}\right| \le RV_{-\beta}(|x|) + \frac{8\nu}{|x|^2} \le RV_{-\beta}(|x|).
        \end{equation}

        Combining \eqref{eq:second_break} with \eqref{eq:approx2}, we get,
        \[
            \left||x|\mathbb{E}S(x, Y) - \frac{x}{|x|}\right| \le RV_{-\beta}(|x|).
        \]
        This proves the result.
    \end{proof}

    \begin{proof}[Proof of Lemma \ref{lemma:score_bound}]
        It follows by combining Lemmas \ref{lemma:concentration} and \ref{lemma:error_expectation} and by the monotone equivalence of regularly varying functions \cite[][Theorem 1.5.3]{bingham_regular_1987}.
    \end{proof}

\subsection{Proof of Theorem~\ref{theo:score_bound}}    
\label{sec:score_proof}  
   \begin{proof}[Proof of Theorem~\ref{theo:score_bound}]
        Define $g:(\mathbb{R}^d)^n \to \mathbb{R}$ by,
        \[
            g(y_1, \dots, y_n) = \sup_{\{x : |x - y_n|\le 2|y_n|/n\}} \left|\sum_{j=1}^{n-1} S(x, y_j) - \frac{(n-1) x}{|x|^2}\right|.
        \]
        Then, $g$ is symmetric in its first $n-1$ arguments. Let $\{\tilde{Y}_{n,1}, \dots, \tilde{Y}_{n,n-1}, Z_n\}$ be as defined in Section \ref{sec:order_stats}. Observe that, 
        \begin{align*}
            \sup_{x \in \mathbb{R}^{d}}\left|S_n(x) - \hat{S}_n(x)\right| &= \sup_{x \in \mathcal{F}_n}\left|\sum_{j\neq n-k} S(x, Y_{(j)}) - \frac{(n-1)x}{|x|^2}\right|\\
            &= g(Y_{-(n-k)}, Y_{(n-k)}).
        \end{align*}
        Then, by Proposition \ref{propo:rejection}, 
        \begin{align*}
            &\PP\left(\sup_{x \in \mathbb{R}^{d}}\left|S_n(x) - \hat{S}_n(x)\right| > n^{1-1/\beta-\delta} \mid A_n \right) \\
            &\quad \quad \quad = \PP\left(g(Y_{-(n-k)}, Y_{(n-k)}) > n^{1 - 1/\beta -\delta} \mid |Y_{(n-k)}| > 2n\sqrt{\nu}\right) \\
            &\quad \quad \quad \le 4\sqrt{k+1} \mathbb{Q}\left(g(\tilde{Y}_{n,1}, \dots, \tilde{Y}_{n,n-1}, Z_n) > n^{1 - 1/\beta -\delta} \mid |Z_{n}| > 2n\sqrt{\nu}\right).
        \end{align*}
        Once again, we will show that the right-hand side goes to $0$ as $n \to \infty$. Towards that end, observe that,
        \begin{align*}
            &n^{1/\beta - 1 + \delta} g(\tilde{Y}_{n,1}, \dots, \tilde{Y}_{n,n}, Z_n) \\
            &\quad \quad = n^{1/\beta-1+\delta}\sup_{|x - Z_n| \le 2|Z_n|/n} \left|\sum_{j=1}^{n-1}S(x, Y_{n,j}) - \frac{(n-1)x}{|x|^2}\right|\\
            &\quad \quad \le n^{1/\beta - 1 + \delta}\frac{n-1}{2|Z_n|}\sup_{|x - Z_n| \le 2|Z_n|/(n-1)} \frac{2|Z_n|}{n-1}\left|\sum_{j=1}^{n-1}S(x, Y_{n,j}) - \frac{(n-1)x}{|x|^2}\right|\\
            &\quad \quad \le n^{\delta}\left(\frac{n^{1/\beta}}{|Z_n|}\right)\left(\frac{n-1}{2n}\right)\mathbb{S}_{n-1}(Z_n),
        \end{align*}
        where $\mathbb{S}_{n-1}(Z_n)$ is as defined in \eqref{eq:boldS}. Due to Proposition \ref{propo:evt}, $n^{1/\beta}/|Z_n|$ converges to a non-degenerate random variable, say $G$. Also, the middle term converges to $1/2$.

        Consider $\mathbb{S}_{n-1}(Z_n)$. Choose $0 < \epsilon < 1/2 - \delta$. Such an $\epsilon$ always exists since $\delta < 1/2$. Let $T_0(\epsilon)$ be the $T_0(\epsilon)$ of Lemma \ref{lemma:score_bound}. For all $n \ge T_0(\epsilon)$, define the event $A_n = \{|Z_n| > 2n\sqrt{\nu}\}$. Then, on the event $A_n$, by Lemma \ref{lemma:score_bound},
        \begin{align*}
            n^{\delta}\mathbb{S}_{n-1}(Z_n) &\le 48\nu n^{2\epsilon + \delta-1} + \frac{|Z_n|^{-\beta}}{n^{\epsilon(d-1) - \delta}}L'(|Z_n|) + n^{\delta}g(|Z_n|) + n^{\delta}R_n(|Z_n|).
        \end{align*}
        where, $L', g,$ and $R_n$ are as in the statement of Lemma \ref{lemma:score_bound}. Remember that $Z_n$ scales as $n^{1/\beta}$ and $\delta < 1/2$. And so, the second and third terms on the right-hand side goes to $0$ as $n \to \infty$ for any choice of $\epsilon$ by the continuous mapping theorem. The first term goes to $0$ because of the choice of $\epsilon$. Also, from the statement of Lemma \ref{lemma:score_bound},  we have,
        \begin{align*}
            \mathbb{Q}\left(R_n(|Z_n|) > tn^{-\delta} \mid |Z_n| \right) \le 5^d\exp\left(-\frac{n^{1-2\epsilon - 2\delta}t^2}{512}\right) + C_d\frac{5^d|Z_n|^{-\beta}}{2^dn^{d\epsilon-1}}L(|Z_n|)
        \end{align*}
        where $L(t) = t^{\beta+d}h(t)$ converges to a finite constant as $n \to \infty$. Once again, since $|Z_n|$ scales as $n^{1/\beta}$, the right-most term goes to $0$ for any $\epsilon > 0$. Additionally, the exponential term on the right-hand side goes $0$ for all $t>0$ since $\epsilon + \delta < 1/2$.
        
        The above arguments imply that $n^{\delta}\mathbb{S}_{n-1}(Z_n)$ conditioned on the event $A_n$ goes to $0$ in probability. This implies that
        \[
            \mathbb{Q}\left(g(\tilde{Y}_{n,1}, \dots, \tilde{Y}_{n,n-1}, Z_n) > n^{1 - 1/\beta - \delta} \mid |Z_{n}| > 2n\sqrt{\nu}\right) \longrightarrow 0.
        \]
        As a consequence,
        \[
           \PP\left(\sup_{x \in \mathbb{R}^{d}}\left|S_n(x) - \hat{S}_n(x)\right| > n^{1-1/\beta-\delta} \mid A_n \right) \longrightarrow 0
        \]
        as $n \to \infty$.
    \end{proof}
    \section{Proof of Lemma~\ref{lemma:excursion_length}}
\label{sec:exit_technical}

    In this section, we will prove Lemma~\ref{lemma:excursion_length}. First, we will control the right excursions.

    \begin{lemma}
        Let $y \in [x_n^{+}, Y_{(n)})$ be arbitrary but fixed. Define $T_y^{\uparrow}$ to be the random variable such that for all $t > 0$
        \[
            \PP(T_y^{\uparrow} > t) = \exp\left(-\int_0^t \lambda_n(y + u, +1)\ du\right).
        \]
        Then, on the event $E_n$ for large $n$
        \begin{equation}
        \label{eq:right_excursion_1}
            \EE(T_y^{\uparrow}) \le \sqrt{\nu} + \frac{\sqrt{\nu\pi}}{2}\frac{\Gamma(\nu/2)}{\Gamma((\nu + 1)/2)} =: C_{\nu}.
        \end{equation}
        Suppose either $\gamma_n(x) \equiv 0$ or defined by \eqref{eq:excess_rate}. Let $0 < c < 2\sqrt{\nu}$ be a constant. Then, on the event $E_n$ for large $n$
        \begin{equation}
        \label{eq:right_excursion_lower}
            \EE(T_y^{\uparrow}) \ge ce^{-(\nu + 1)}\left(1-2c(\nu + 1)\overline{\gamma_n}\right).
        \end{equation}
    \end{lemma}
    \begin{proof}
        From \eqref{eq:excess_cases}, it follows that for $x > x_n^{+}$,
        \[
            \lambda_n(x, +1) = (\nu + 1)S_n(x) + \gamma_n(x) \ge (\nu + 1)(S(x, Y_{(n)}))_{+}.
        \]
        Let $y \ge x_n^{+}$ be arbitrary but fixed. Then, for all $t> 0$
        \[
            \PP(T_y^{\uparrow} > t) \le \exp\left(-(\nu + 1)\int_0^t (S(y + u, Y_{(n)}))_{+}\ du\right).
        \]
        And so,
        \[
            \EE(T_y^{\uparrow}) = \int_{0}^{\infty} \PP(T_y^{\uparrow} > t)\ dt \le \int_0^{\infty} \exp\left(-\int_0^t \frac{(\nu + 1)(y + u - Y_{(n)})_{+}}{\nu + (y + u - Y_{(n)})^2}\ du\right) \ dt.
        \]
        Solving the integral on the right-hand side gives, for $y \in [x_n^{+}, Y_{(n)})$
        \[
            \EE(T_y^{\uparrow}) \le (Y_{(n)} - y) + \frac{\sqrt{\nu\pi}}{2}\frac{\Gamma(\nu/2)}{\Gamma((\nu + 1)/2)} \le \sqrt{\nu} + \frac{\sqrt{\nu\pi}}{2}\frac{\Gamma(\nu/2)}{\Gamma((\nu + 1)/2)}, 
        \]
        since $Y_{(n)} - y \le Y_{(n)} - x_n^{+} \le \sqrt{\nu}$.

        For the lower bound, let $0 < c < 2\sqrt{\nu}$ be a constant. Then, on the event $E_n$, $Y_{(n)} + c \le Y_{(n)} + Y_{(n)}/n$ and for $u \in [x_n^{+}, Y_{(n)} + c]$, $S_n(x) < \hat{S}_n(x) + n^{1-1/\beta-\delta}$. Since $\gamma_n \le (\nu + 1)\overline{\gamma_n}$, we get, for $u > x_n^{+}$
        \begin{align*}
            \lambda_n(u, +1) = (\nu + 1)S_n(u) + \gamma_n &\le (\nu + 1)\left(\hat{S}_n(u) + n^{1-1/\beta-\delta} + \overline{\gamma_n}\right) \\
            &\le (\nu + 1)\left(\frac{n-1}{x_n^{+}} + \frac{1}{2\sqrt{\nu}} + n^{1-1/\beta-\delta} + \overline{\gamma_n}\right).
        \end{align*}      
        This, together with $x_n^{+} > x_n^{-}$ and \eqref{eq:gamma_bound} gives 
        \[
            \lambda_n(u, + 1) \le (\nu + 1)\left(\frac{1}{2\sqrt{\nu}} + 2\overline{\gamma_n}\right).
        \]
        For arbitrary $y \in (x_n^{+}, Y_{(n)}]$, Markov's inequality implies,
        \begin{align*}
            \EE(T_{y}^{\uparrow}) \ge c \PP(T_y^{\uparrow} > c) &= c\exp\left(-\int_{y}^{y + c}\lambda_n(u,+1)\ du\right) \\
            &\ge c\exp\left(-c(\nu + 1)\left(\frac{1}{2\sqrt{\nu}} + 2\overline{\gamma_n}\right)\right)\\
            &\ge ce^{-(\nu + 1)}\left(1-2c(\nu + 1)\overline{\gamma_n}\right),
        \end{align*}
        since $c < 2\sqrt{\nu}$ and $cn^{1-1/\beta-\delta} < 1$ for large $n$, and $e^{-x} \ge 1 - x$. Hence, proved.
    \end{proof}

    \begin{lemma}
    \label{lemma:right_excursion_bounds}
        Suppose either $\gamma_n(x) \equiv 0$ or defined by \eqref{eq:excess_rate}. Let $y \in [x_n^{+}, Y_{(n)})$ be arbitrary but fixed and $0 < c < 2\sqrt{\nu}$ be a constant. Consider $T_n(y, +1)$ where $T_n(y, v)$ is defined as in \eqref{eq:defns}. Then, for large $n$, on the event $E_n$,
        \begin{align*}
            &2ce^{-(\nu + 1)}\left(1-2c(\nu + 1)\overline{\gamma_n}\right)  \le \EE(T_n(y, +1)) \le \frac{2\sqrt{\nu} + C_{\nu}} {\tilde{p}_n} + \sqrt{\nu},
        \end{align*}
        where $1 \ge \tilde{p}_n \ge \exp(- \sqrt{\nu}(\nu + 1)\overline{\gamma_n})$.
    \end{lemma}

    \begin{proof}
        Let $y \in [x_n^{+}, Y_{(n)})$ be arbitrary but fixed. We are interested in $\EE(T_n(y, +1))$. Starting at $(y, +1)$, let $T_y^{\uparrow}$ be the distance traveled by the process to the right of $y$ before the first jump. Then, it follows that,
        \[
            \EE(T_n(y, +1)) = \EE\left(T_y^{\uparrow} + \EE(T_n(y + T_y^{\uparrow}, -1) \mid T_y^{\uparrow})\right).
        \]
        Here, $T_y^{\uparrow}$ is distributed such that,
        \[
            \PP(T_y^{\uparrow} > t) = \exp\left(-\int_0^t \lambda_n(y + u, 1)\ du\right), \quad t \ge 0.
        \]
        Now, note that $S_n(x) > 0$ for all $x \ge Y_{(n)}$. Moreover, since $\gamma_n(x) \equiv 0$ for $x \ge Y_{(n)}$, we get $\lambda_n(x, -1) = 0$ for $x \ge Y_{(n)}$. This means that if $y + T_y^{\uparrow} > Y_{(n)}$, there will be no jump till the process drops back to $Y_{(n)}$. Thus, we get for any $t > 0$,
        \begin{align*}
            2t \le t + T_n(y + t, -1) \le t + (y + t - Y_{(n)})_{+} + T_n(Y_{(n)}, -1) \le 2t + T_n(Y_{(n)}, -1),
        \end{align*}
        since $y < Y_{(n)}$. And so,
        \begin{equation*}
            2 \EE(T_y^{\uparrow}) \le \EE\left(T_n(y, +1)\right) \le 2 \EE(T_y^{\uparrow}) + \EE(T_n(Y_{(n)}, -1)).
        \end{equation*}
        Using \eqref{eq:right_excursion_1} in the above gives,
        \begin{equation}
        \label{eq:right_excursion2}
            2 \EE(T_{y}^{\uparrow}) \le \EE(T_n(y, +1)) \le 2C_{\nu} + \EE(T_n(Y_{(n)}, -1)).
        \end{equation}
        The lower bound then follows from \eqref{eq:right_excursion_lower}. 
        
        To get the upper bound, we need to control $\EE(T_n(Y_{(n)}, -1))$. We will use upper bound in \eqref{eq:right_excursion2}, which is independent of $y$. Consider $T_n(Y_{(n)}, -1)$, i.e. starting from $(Y_{(n)}, -1)$, the first time to hit $(x_n^{+}, -1)$. Let $\tilde{p}_n$ denote the probability of reaching $(x_n^{+}, -1)$ from $(Y_{(n)}, -1)$ in one-step. Then, with probability $\tilde{p}_n$, $T_n(x_n^{+}, -1) = (Y_{(n)} - x_n^{+})$. On the other hand, with probability $(1 - \tilde{p}_n)$, the process travels a $m_n < (Y_{(n)} - x_n^{+})$ distance and changes direction. It then takes another $T_n(Y_{(n)} - m_n, +1)$ time to hit $(x_n^{+}, -1)$. But note that $Y_{(n)} - m_n \in (x_n^{+}, Y_{(n)})$. So from \eqref{eq:right_excursion2} we get
        \begin{align*}
            &\EE(T_n(Y_{(n)}, -1))\\
            &\quad = \tilde{p}_n\cdot(Y_{(n)} - x_n^{+}) + (1 - \tilde{p}_{n})  \EE\left(m_n + \EE(T_n(Y_{(n)} -m_n, +1)) \mid m_n < Y_{(n)} - x_n^{+}\right)\\
            &\quad \le  (Y_{(n)} - x_n^{+}) + (1 - \tilde{p}_n)\left((Y_{(n)} - x_n^{+}) + 2C_{\nu} + \EE(T_n(Y_{(n)}, -1))\right).
        \end{align*}
        This, together with \eqref{eq:right_excursion2} gives, 
        \begin{equation}
        \label{eq:right_excursion3}
            \EE(T_n(y, +1)) \le \left(\frac{2} {\tilde{p}_n} + 1 \right) (Y_{(n)} - x_n^{+}) + \frac{1}{\tilde{p}_n} C_{\nu} \le \frac{2\sqrt{\nu} + C_{\nu}} {\tilde{p}_n} + \sqrt{\nu},
        \end{equation}
        since $(Y_{(n)} - x_n^{+}) < \sqrt{\nu}$.
        
        Finally, consider $\tilde{p}_n$, the probability of hitting $(x_n^{+}, -1)$ from $(Y_{(n)}, -1)$ in one step. Since $S_n(x) > 0$ for $x > x_n^{+}$, for $u \in (x_n^{+}, Y_{(n)}]$, $\lambda_n(u, -1) = (\nu + 1)(-S_n(u))_{+} + \gamma_n(u) \le (\nu + 1)\overline{\gamma_n}$. Then,
        \begin{align*}
            \tilde{p}_n = \exp\left(-\int_{x_n^{+}}^{Y_{(n)}} \lambda_n(u, -1)\ du\right) \ge \exp\left(-(\nu + 1)(Y_{(n)} - x_n^{+})\overline{\gamma_n}\right) \ge 1 - \sqrt{\nu}(\nu + 1)\overline{\gamma_n},
        \end{align*}
        since $Y_{(n)} - x_n^{+} \le \sqrt{\nu}$. 
        Hence, proved.
    \end{proof}

    \begin{lemma}
    \label{lemma:right_both}
        Suppose either $\gamma_n(x) \equiv 0$ or defined by \eqref{eq:excess_rate}. For large $n$, on the event $E_n$, there exist constants $C_1$ and $C_2$ such that
        \[
            C_1 + O_{p}(n^{1-1/\beta}) \le \EE(T_{n}(x_n^{+}, +1)) \le C_2 + O_p(n^{1-1/\beta})
        \]
        as $n \to \infty$.
    \end{lemma}
    \begin{proof}
        Recall that $\overline{\gamma_n} = n^{1-1/\beta}/(G_k + c_{\beta})$. From Lemma~\ref{lemma:right_excursion_bounds} we get,
        \begin{align*}
            \EE(T_n(x_n^{+}, +1)) &\ge 2ce^{-(\nu + 1)}\left(1-2c(\nu + 1)\overline{\gamma_n}\right).
        \end{align*}
        And so the lower bound follows. For the upper bound, Lemma~\ref{lemma:right_excursion_bounds} gives,
        \begin{align*}
            \EE(T_n(x_n^{+}, +1)) \le \sqrt{\nu} + (2\sqrt{\nu} + C_{\nu})\exp(\sqrt{\nu}(\nu + 1)\overline{\gamma_n}).
        \end{align*}
        The result follows trivially for $\beta =1$. For $\beta < 1$, $\overline{\gamma_n} \to 0$ and the result follows from the upper bound $e^x < 1 + 2x$ for small $x$. 
    \end{proof}

    Lemma~\ref{lemma:right_both} proves that the right excursions remain $O_p(1)$ as $n \to \infty$ for both the canonical ZZP and subsampling ZZP. We now prove $O_{p}(1)$ upper bound on $\eta_n$ for canonical and subsampling ZZP separately. First, consider the canonical case.

    \begin{lemma}
    \label{lemma:eta_canonical}
        Suppose $\gamma_n \equiv 0$. For large $n$, on the event $E_n$, there exists a constant $C_3$ such that
        \[
            \EE(\eta_n) \le C_3 + O_p(n^{1-1/\beta})
        \]
        as $n \to \infty$.
    \end{lemma}
    \begin{proof}
        Let $T_n^{\downarrow}$ be the time to first jump for the process started at $(x_n^{-}, -1)$. If $T_n^{\downarrow} < W_n$, the process jumps before hitting $x_n^{-}$. For large $n$, on the event $E_n$, $S_n(x) < 0$ for $x \in (x_n^{-}, x_n^{+})$. When $\gamma \equiv 0$, there is no other jump before the process hits $(x_n^{+}, 1)$. So,
        \[
            \eta_n = \begin{cases}
                2T_n^{\downarrow}, & T_n^{\downarrow} < W_n, \\
                W_n, & T_n^{\downarrow} \ge W_n.
            \end{cases}
        \]
        which implies,
        \[
            \EE(\eta_n) = 2 \EE(\min\{T_n^{\downarrow}, W_n\}) - W_n \PP(T_n^{\downarrow} \ge W_n)  = 2\int_0^{W_n} \PP(T_n^{\downarrow} > t)\ dt - W_n p_{\tau, n}.
        \]
        However, from Lemma~\ref{lem:pn_magnitude} and Theorem~\ref{theo:width}, $W_n p_{\tau, n} = O_p(n^{(1-1/\beta)\nu})$. So we only need to focus on the first term.

        Consider,
        \begin{align*}
            \int_0^{W_n} \PP(T_n^{\downarrow} > t)\ dt  &= \int_0^{W_n} \exp\left(-\int_{0}^t \lambda_n(x_n^{+} - u, -1)\ du\right)\ dt\\ &= \int_0^{W_n} \exp\left((\nu + 1)\int_{x_n^{+} - t}^{x_n^{+}} S_n(u)\ du\right)\ dt
        \end{align*}
        since $\lambda_n(x,-1) = (\nu + 1)(-S_n(x))_{+}$ and $S_n(x) \le 0$ for $x \in (x_n^{-}, x_n^{+})$. But also,
        \begin{align*}
            &\exp\left((\nu + 1)\int^{x_n^{+}}_{x_n^{+} - t} S_n(u) du\right)\\
             &\quad \quad = \left(\frac{\nu + (x_n^{+} - Y_{(n)})^2}{\nu + (x_n^{+} - t - Y_{(n)})^2}\right)^{(\nu + 1)/2}\cdot \prod_{j=1}^{n-1}\left(\frac{\nu + (x_n^{+} - Y_{(j)})^2}{\nu + (x_n^{+} - t - Y_{(j)})^2}\right)^{(\nu + 1)/2}\\
             &\quad \quad \le \left(\frac{\nu + (x_n^{+} - Y_{(n)})^2}{\nu + (x_n^{+} - t - Y_{(n)})^2}\right)^{(\nu + 1)/2}\cdot \prod_{j=1}^{n-1}\left(\frac{\nu + (x_n^{+} - Y_{(j)})^2}{\nu + (x_n^{-} - Y_{(j)})^2}\right)^{(\nu + 1)/2}\\
             &\quad \quad = p_n\left(\frac{\nu + (x_n^{-} - Y_{(n)})^2}{\nu + (x_n^{+} - t - Y_{(n)})^2}\right)^{(\nu + 1)/2},
        \end{align*}
        where recall that $p_n = \pi^{(n)}(x_n^{-})/\pi^{(n)}(x_n^{+})$. We get,
        \begin{align*}
            \int_0^{W_n} \PP(T_n^{\downarrow} > t)\ dt &\le p_n\left(\nu + (x_n^{-} - Y_{(n)})^2\right)^{(\nu + 1)/2}K(\nu)\\ 
            &= K(\nu)\frac{p_n}{(Y_n - x_n^{-})^{-(\nu + 1)}}\left(1 + \frac{\nu}{(Y_n - x_n^{-})^{2}}\right)^{(\nu + 1)/2}
        \end{align*}
        for some constant $K(\nu)$. The result then follows from Theorem~\ref{theo:width} and Lemma~\ref{lem:pn_magnitude}.
    \end{proof}

    \begin{lemma}
    \label{lemma:eta_subsampling}
        Suppose $\gamma_n$ is as defined in \eqref{eq:excess_rate}. For large $n$, on the event $E_n$, there exists a constant $C_3$ such that
        \[
            \EE(\eta_n) \le C_3 + O_p(n^{1-1/\beta})
        \]
        as $n \to \infty$.
    \end{lemma}

    \begin{proof}
        We will follow the argument in \cite{monmarche_piecewise_2016}. Let $Z_0 = (X_0, V_0) = (x^{+}_n, -1)$ and $T_0 = 0$. Let $(F_k)_{k \ge 1}$ be an iid sequence of Exponential$(1)$ random variables. Suppose $Z_k = (X_k, V_k)$ and $T_k$ has been defined for some $k \ge 0$. Define $\tau_{k+1}$ by
        \[
            F_{k+1} = \int_{0}^{\tau_{k+1}} \lambda_n\left(X_k + V_ks, V_k\right)ds,
        \]
        where $\lambda_n(x, v) = (\nu + 1)(vS_n(x))_{+} + \gamma_n(x)$. Let $T_{k+1} = T_k + \tau_{k+1}$. Define $X_{k+1} = X_k + T_{k+1}\cdot V_k$ and $V_{k+1} = -V_k$. Repeat this procedure to get a skeleton $(Z_k, T_k)_{k\ge 0} = ((X_k, V_k), T_k)_{k \ge 0}$. Define a continuous time process $Z^n_t = (X^n_t, V^n_t)$ with initial condition $(X^n_0, V^n_0) = (x^{+}_n, -1)$ by,
        \[
            X^n_t = X_{k} + (t - T_k)\cdot V_k, \quad V^n_t = V_k; \quad t \in (T_{k}, T_{k+1}].
        \]
        Then, $Z^n_t$ is a Zig-Zag process with rate $\lambda_n(x, v)$. By construction, $V_{2k-1} = -V_{2k} = 1$ i.e. starting at $(x_n^{+}, -1)$, the process alternates between left moves and right moves. Moreover, times $\tau_{2k-1}$ and $W\tau_{2k}$ indicate the duration (and hence length) of these moves, respectively. 
        
        Consider $\eta_n$. Observe that $\eta_n$ is bounded above by twice the distance traveled to the left. Indeed, if $X^n_{\eta} = x_n^{+}$, the process spends as much time going left as going right, and if $X^n_{\eta} = x_n^{-}$, then it spends more time going left than right. Let $N_{\eta}$ be the number of left moves before time $\eta_n$. Then,
        \[
            \eta_n \le 2\sum_{k = 1}^{N_{\eta}} \tau_{2k-1} \quad \text{ where } \quad  N_{\eta} = \max\{k: T_{2k-1} < \eta_n\}.
        \]
        But for all $k \le N_{\eta}$, $\tau_{2k-1} \le W_n$. And so, $\eta_n \le 2W_nN_{\eta}$. 
        
        Now, for large $n$, on the event $E_n$, $S_n(x) \le 0$ for all $x \in (x_n^{-}, x_n^{+})$. Then, $\lambda_n(x,+1) = \gamma_n(x)$ for all $x \in (x_n^{-}, x_n^{+})$. 
        Define,
        \[
            J_{k} = \left\{F_{2k} < \int_{x_n^{-}}^{x_n^{+}}\gamma_n(u)\ du\right\}, \quad N= \min\{k\ge 1, J^c_{k}\}.
        \]
        Then, $N$ is a geometric random variable with \[
            \EE(N) = 1/\PP(J^c_1) \le e^{(\nu + 1)W_n\overline{\gamma_n}}.
        \]
        Since $\gamma_n \ge 0$, on the event $J^{c}_{k}$, we have that $\tau_{2k} > W_n$. Hence, if $T_{2k-1} < \eta_n$ and $J^c_k$ happens, then there will be no other jump before the process hits $(x_n^{+}, +1)$. And so, it must be that $N_{\eta} \le N$. Combining this with the earlier observation, we get,
        \[
            \eta_n \le 2W_nN.
        \]
        When $\beta = 1$, taking expectation on both sides of the above equation, we get,
        \[
            \EE(\eta_n) \le 2W_ne^{(\nu + 1)W_n\overline{\gamma_n}}.
        \]
        The result follows for from Theorem~\ref{theo:width} and equation \eqref{eq:w_times_gamma}.

        Now, suppose $0 < \beta < 1$. Then, from Theorem~\ref{theo:width}, $W_n \to \infty$ as $n \to \infty$. Let $n$ be large enough such that $W_n > 1$. Let $0 < \delta_n = W_n^{1/(1 + \nu)} < W_n$ on the event $E_n$. Let $B$ be the event $\{W_1 < \delta_n, W_1 < F_2/(\nu+1)\overline{\gamma_n}\}$. Then, on the event $B$, since $W_1 < \delta_n < W_n$, the process does not exit to the left in the first step. However, also on the event $B$, it follows that, $W_2 > W_1$. To see this, suppose $W_2 \le W_1$. Then, since $\lambda_n(x,+1) = \gamma_n(x) \le (\nu+1)\overline{\gamma_n}$ for $x \in (x_n^{-}, x_n^{+})$, we get that,
        \begin{align*}
             F_2 &= \int_{0}^{W_2} \lambda_n(x_n^{+} - W_1 + u, +1)\ du \le \int_{0}^{W_1} \gamma_n(x_n^{+} - W_1 + u)\ du  \le (\nu+1)W_1\overline{\gamma_n},
        \end{align*}
        which is a contradiction. So, $W_2 > W_1$. Thus, on the event $B$, the process exits on the right in the first step and hence, $\eta_n = 2W_1$ on the event $B$. We have,
        \begin{align*}
            \EE(\eta_n) = \EE(1_{B}\eta_n) + \EE(1_{B^c}\eta_n) \le 2\EE(1_{B}W_1) + 2W_n\EE(1_{B^c}N).
        \end{align*}
        But also, on the event $B^c$, if $F_2 < (\nu+1)\delta_n{\overline{\gamma_n}}$ then, $N > 1$. We thus get,
        \begin{align*}
            \EE(\eta_n) &\le 2\EE(1_{B}W_1) + 2W_n\PP(B^c)(1 + \EE(N)) \\
                &\le 2\EE(1_{B}W_1) + 2W_n\PP(B^c)\left(1 + e^{(\nu+1)W_n\overline{\gamma_n}}\right).            
        \end{align*}
        
        First consider,
        \begin{align*}
            \PP(B) = \PP(W_1 < \delta_n, W_1 < F_2/(\nu+1)\overline{\gamma_n}) &= \EE\left(e^{-(\nu+1)W_1\overline{\gamma_n}}1(W_1 < \delta_n)\right)
            \\
            &\ge \EE\left((1 - (\nu+1)W_1\overline{\gamma_n})1(W_1 < \delta_n)\right)\\
            &= \PP(W_1 < \delta_n) - (\nu+1)\overline{\gamma_n}\EE(W_11(W_1 < \delta_n)). 
        \end{align*}
        Also, note that $\EE(1_{B}W_1) \le \EE(W_11(W_1 < \delta_n))$. And so, we get, after rearranging
        \begin{align*}
            \EE(\eta_n) 
            &= 2W_n \PP(W_1 > \delta_n)(1 + e^{(\nu +1)W_n\overline{\gamma_n}}) \\
            &\quad \quad + 2\left(1 +  (\nu +1)W_n\overline{\gamma_n}(1 + e^{(\nu +1)W_n\overline{\gamma_n}})\right)\EE(W_11(W_1 < \delta_n))
        \end{align*}
        From \eqref{eq:w_times_gamma}, on the set $E_n$, for large $n$, 
        \begin{align*}
            W_n\overline{\gamma_n} \le 4 + n^{1-1/\beta}\sqrt{\nu}/G_k,
        \end{align*}
        since $c_{\beta} = 0$ for $\beta < 1$. The right-hand side goes to $4$ as $n \to \infty$. Hence, for large $n$, $W_n\overline{\gamma_n} \le 5$ on the event $E_n$. We get,
        \begin{align*}
            \EE(\eta_n)
            &\le 2W_n \PP(W_1 > \delta_n)(1 + e^{5(\nu +1)}) + 2(1 + 5(\nu + 1)(1 + e^{5(\nu +1)}))\EE(W_11(W_1 < \delta_n))
        \end{align*}
        
        Now, observe that for large $n$, on the event $E_n$, we have $\lambda_n(x, -1) = -(\nu + 1)S(x, Y_{(n)})$ for all $x\in(x_n^{-}, x_n^{+})$. Let $t \le W_n$. Then,
        \begin{align*}
            \PP(W_1 > t) = \PP\left(F_1 > \int_{0}^t\lambda_n(x_n^{+} - u, -1)\ du \right) = \left(\frac{\nu + (Y_{(n)} - x_n^{+})^2}{\nu + (Y_{(n)} - x_n^{+} + t)^2}\right)^{(\nu + 1)/2}
        \end{align*}
        Then, recalling that $\delta_n = W_n^{1/(1 + \nu)}$, we get,
        \[
            W_n \PP\left(W_1 > W_n^{1/(1+\nu)}\right) \le W_n\left(\frac{\nu + (Y_{(n)} - x_n^{+})^2}{W_n^{2/(1 + \nu)}}\right)^{(\nu + 1)/2} = \left(\nu + (Y_{(n)} - x_n^{+})^2\right)^{(\nu + 1)/2}.
        \]
        Moreover, 
        \begin{align*}
            \EE(W_11(W_1 < \delta_n)) &\le \int_{0}^{\delta_n} \PP(W_1 > t)dt  \\
            &\le (\nu + (Y_{(n)} - x_n^{+})^2)^{(\nu + 1)/2}\int_0^{\infty}(\nu + t^2)^{-(\nu + 1)/2}\ dt\\
            &\le K(\nu)(\nu + (Y_{(n)} - x_n^{+})^2)^{(\nu + 1)/2},
        \end{align*}
        for some constant $K(\nu)$. Combining these observations and recalling that $|Y_{(n)} - x_n^{+}| \le \sqrt{\nu}$ gives,
        \[
            \EE(\eta_n) \le (2K(\nu) + (2 + 10(\nu + 1)K(\nu))(1 + e^{(\nu +1)5}))\left(2\nu\right)^{(\nu + 1)/2}.        
        \]
        The result follows.        
    \end{proof}

    \begin{proof}[Proof of Lemma~\ref{lemma:excursion_length}]
        The result follows by combining Lemma~\ref{lemma:right_both}, Lemma~\ref{lemma:eta_canonical}, and Lemma~\ref{lemma:eta_subsampling}.
    \end{proof}

\end{appendix}

%%%%%%%%%%%%%%%%%%%%%%%%%%%%%%%%%%%%%%%%%%%%%%
%% Support information, if any,             %%
%% should be provided in the                %%
%% Acknowledgements section.                %%
%%%%%%%%%%%%%%%%%%%%%%%%%%%%%%%%%%%%%%%%%%%%%%
\begin{acks}[Acknowledgments]
The authors would like to thank Alberto Bordino for helpful discussions on uniform law of large numbers.
\end{acks}
%%%%%%%%%%%%%%%%%%%%%%%%%%%%%%%%%%%%%%%%%%%%%%
%% Funding information, if any,             %%
%% should be provided in the                %%
%% funding section.                         %%
%%%%%%%%%%%%%%%%%%%%%%%%%%%%%%%%%%%%%%%%%%%%%%
\begin{funding}
SA was supported by the EPSRC studentship grant EP/W523793/1, and by the Independent Research Fund Denmark (DFF) through the Sapere Aude Starting Grant (5251-00032B). SG acknowledges support from the European Union (ERC), through the Starting Grant `PrSc-HDBayLe', project number 101076564. GOR was supported by PINCODE (EP/X028119/1), EP/V009478/1 and by the UKRI grant, OCEAN, EP/Y014650/1.
\end{funding}

%%%%%%%%%%%%%%%%%%%%%%%%%%%%%%%%%%%%%%%%%%%%%%
%% Supplementary Material, including data   %%
%% sets and code, should be provided in     %%
%% {supplement} environment with title      %%
%% and short description. It cannot be      %%
%% available exclusively as external link.  %%
%% All Supplementary Material must be       %%
%% available to the reader on Project       %%
%% Euclid with the published article.       %%
%%%%%%%%%%%%%%%%%%%%%%%%%%%%%%%%%%%%%%%%%%%%%%
%\begin{supplement}
%\stitle{???}
%\sdescription{???.}
%\end{supplement}

%%  MR numbers will be added by VTeX.                      %%
%%                                                         %%
%%  Use \cite{...} to cite references in text.             %%
%%                                                         %%
%%%%%%%%%%%%%%%%%%%%%%%%%%%%%%%%%%%%%%%%%%%%%%%%%%%%%%%%%%%%%

% if your bibliography is in bibtex format, uncomment commands:
\bibliographystyle{imsart-number} % Style BST file (imsart-number.bst or imsart-nameyear.bst)
\bibliography{references}       % Bibliography file (usually '*.bib')

%% or include bibliography directly:
% \begin{thebibliography}{}
% \bibitem{b1}
% \end{thebibliography}

\end{document}